\documentclass[en, amspaper, secnum, nobib]{./cls/mymcls}
\usepackage{stmaryrd}
\title{Gelfand-type theorems for dynamical Banach modules}

\makeatletter
\def\moverlay{\mathpalette\mov@rlay}
\def\mov@rlay#1#2{\leavevmode\vtop{%
   \baselineskip\z@skip \lineskiplimit-\maxdimen
   \ialign{\hfil$\m@th#1##$\hfil\cr#2\crcr}}}
\newcommand{\charfusion}[3][\mathord]{
    #1{\ifx#1\mathop\vphantom{#2}\fi
        \mathpalette\mov@rlay{#2\cr#3}
      }
    \ifx#1\mathop\expandafter\displaylimits\fi}
\makeatother

\newcommand{\cupdot}{\charfusion[\mathbin]{\cup}{\cdot}}
\newcommand{\bigcupdot}{\charfusion[\mathop]{\bigcup}{\cdot}}

\begin{document}

\begin{abstract}
  The representation theorems of Gelfand and Kakutani for commutative C*-algebras and AM- and AL-spaces are the basis for the Koopman linearization of topological and measure-preserving dynamical systems. In this article we prove versions of these results for dynamics on topological and measurable Banach bundles and the corresponding weighted Koopman representations on Banach modules.\medskip
\end{abstract}

\maketitle
\section{Introduction}
The concept of \emph{Koopman linearization} provides a very powerful method to study dynamical systems, see \cite{EFHN2015}. Given a \emph{topological $G$-dynamical system}, i.e., a locally compact group $G$ acting continuously on a locally compact space $\Omega$, one can consider the induced \emph{Koopman representation} of $G$ as automorphisms of the commutative C*-algebra $\mathrm{C}_0(\Omega)$ of all continuous functions on $\Omega$ vanishing at infinity given by $T(g)f(x) \defeq f(g^{-1}x)$ for $x \in \Omega$, $f \in \mathrm{C}_0(\Omega)$, and $g \in G$.\\
Passing to these linear operators opens the door for the use of functional analytic tools (e.g., spectral theory) to investigate the qualitative properties of the $G$-dynamical system. This is justified by Gelfand's representation theory which shows that no relevant information is lost in this process.\\
More precisely and in terms of category theory (see \cite{MacL1998} for an introduction), assigning the Koopman representation to a group action defines an equivalence of the category of topological $G$-dynamical systems and the category of strongly continuous representations of $G$ as automorphisms of commutative C*-algebras, see, e.g., Section 1.4 of \cite{Dixm1977} and Sections 4.3 and 4.4 of \cite{EFHN2015}.\\
Likewise, in the measure theoretic setting Koopman representations on $\mathrm{L}^1$-spaces reflect the qualitative behavior of measure-preserving systems up to null sets (under some separability assumptions, see Section 7.3 and Chapter 12 of \cite{EFHN2015}). Using Kakutani's representation theorem for AL-spaces (see Theorem II.8.5 of \cite{Scha1974}), such Koopman representations can also be characterized in terms of Banach lattice theory.\\
These results established a connection between topological dynamics and ergodic theory on one hand and functional analysis and operator theory on the other, leading, amongst others, to the classical and recent ergodic theorems.\\
In this article we prove suitable versions of these representation theorems for dynamics on Banach bundles and modules. This can be the starting point for a systematic operator theoretic investigation of differentiable flows on manifolds and their differentials on tangent bundles.\\


We consider a Banach bundle $E$ over a locally compact or measure space $X$ and dynamics on $E$ compatible with a fixed group action on $X$. These \emph{dynamical Banach bundles} then induce \emph{weighted Koopman representations} on Banach spaces of sections of the bundle.\\
Such dynamical Banach bundles and the induced weighted Koopman representations appear naturally in many contexts. Important examples are so-called evolution families solving nonautonomous Cauchy problems (see Section VI.9 of \cite{EN00}) and derivatives of smooth flows on manifolds (see Chapter 5 of \cite{BaPe2013}).\\
The goal of this article is to characterize such weighted Koopman representations via abstract algebraic and lattice theoretic properties.\\
The correspondence between topological Banach bundles and certain kinds of Banach modules has been established in the 70s and 80s of the last century (see, e.g., \cite{HoKe1977} and \cite{DuGi1983}). We extend these results to a dynamical setting and then also treat the measure theoretic case.\smallskip\\

We start in Section 2 by recalling the concepts of topological and measurable Banach bundles and introduce dynamics on these bundles. Concrete examples motivate the abstract concepts.\\
In the third section we consider Banach modules as the natural operator theoretic counterparts of Banach bundles. We introduce dynamics on these modules and give a first characterization of these operators via a locality condition (see \cref{supportlemma}). In particular, dynamical topological and measurable Banach bundles induce such \enquote{dynamical Banach modules} (see \cref{exampleThom} and \cref{examplemeasurable}).\\

As in the case of Banach lattices (see Sections II.7, II.8 and II.9 of \cite{Scha1974}) there are two important classes of Banach modules which are dual to each other: AM-modules and AL-modules.\\
In Subsection 4.1 we focus on AM-modules, which are known in the literature as (locally) convex Banach modules, see \cite{HoKe1977} or \cite{Gier1998}, and prove our first main result: A Gelfand-type representation theorem for dynamical AM-modules (see \cref{main}). In Subsection 4.2 we then discuss the duality between AM- and AL-modules (see \cref{alvsam}).\\
In Section 5 we see that AM- and AL-modules admit a lattice theoretic structure (see \cref{vectorvaluednorm1} and \cref{alnorm}). In \cref{latticevsmod} and \cref{preparationAL} we show that the algebraic structure of a module and this lattice theoretic structure are strongly related. In particular, weighted Koopman operators can be characterized algebraically (as \emph{weighted module homomorphisms}) or in a lattice theoretic way (as \emph{dominated operators}).\\
We use the lattice theoretic structure to prove our second representation theorem, which clarifies the relation between dynamical measurable Banach bundles and AL-modules (see \cref{main2}). It should be pointed out that---in contrast to the \enquote{AM case}---even the non-dynamical version of this result seems to be new (see \cref{representationAL}).\smallskip\\

In the following all vector spaces are over $\K \in \{\R,\C\}$ and all locally compact spaces are Hausdorff. Moreover, we write $\mathscr{L}(E,F)$ for the space of all bounded linear maps from a normed space $E$ to a normed space $F$.
\section{Dynamical Banach bundles}
\subsection{The topological case}
In this section we define dynamics on topological Banach bundles over some fixed topological dynamical system. Recall the following abstract definition of a Banach bundle (see Definition 1.1 in \cite{DuGi1983}, see also \cite{HoKe1977}).
\begin{definition}
A \emph{(topological) Banach bundle} over a locally compact space $\Omega$ is a pair $(E,p_E)$ consisting of a topological space $E$ and a continuous, open and surjective mapping $p_E\colon E\longrightarrow \Omega$ satisfying the following conditions.
\begin{enumerate}[(i)]
\item Each fiber $E_x \defeq p_E^{-1}(x)$ for $x \in \Omega$ is a Banach space.
\item The mappings
\begin{alignat*}{2}
+ &\colon E \times_\Omega E \longrightarrow E,  \quad &&(u,v) \mapsto u+_{E_{p_E(v)}} v,\\
\cdot \,&\colon \K \times E \longrightarrow E, \quad &&(\lambda,v) \mapsto \lambda \cdot_{E_{p_E(v)}} v
\end{alignat*}
are continuous where $E \times_\Omega E \defeq \bigcup_{x\in \Omega} E_x \times E_x \subset E \times E$ is equipped with the subspace topology.
\item The map
\begin{align*}
\|\cdot\|\colon E \longrightarrow \R_{\geq 0}, \quad v \mapsto \|v\|_{E_{p_E(v)}}
\end{align*}
is upper semicontinuous.
\item For each $x \in \Omega$ and each open set $W \subset E$ containing the zero $0_x \in E_x$ there exist $\varepsilon > 0$ and an open neighborhood $U$ of $x$ such that
\begin{align*}
\left\{v \in p_E^{-1}(U) \mmid \|v\| \leq \varepsilon\right\} \subset W.
\end{align*}
\end{enumerate}
\end{definition}
In the following we usually suppress the mapping $p_E$ and denote the bundle $(E,p_E)$ simply by $E$. Moreover, we call $E$ a \emph{continuous Banach bundle} if the mapping $\|\cdot\|$ is continuous.
\begin{remark}\label{extendedbundle}
Note that if $E$ is a Banach bundle over a locally compact space $\Omega$, we obtain a Banach bundle $\tilde{E}$ over the one-point compactfication $K \defeq \Omega \cupdot \{\infty\}$ in a canonical way by taking the space $\tilde{E} \defeq E \cupdot \{0\}$, the canonical mapping $p_{\tilde{E}} \colon \tilde{E} \longrightarrow K$ and the topology on $\tilde{E}$ generated by the topology on $E$ and the sets
\begin{align*}
U(L,\varepsilon) \defeq \left\{v \in p_{\tilde{E}}^{-1}(\Omega \setminus L) \mmid\|v\| < \varepsilon\right\}
\end{align*}
for compact $L \subset \Omega$ and $\varepsilon >0$. In the following we will frequently make use of this fact.
\end{remark}
We now list some important examples of Banach bundles.
\begin{example}\label{examples1}
\begin{enumerate}[(i)]
\item Let $Z$ be any Banach space and $\Omega$ a locally compact space. Then $E \defeq \Omega \times Z$ is a continuous Banach bundle over $\Omega$, called the \textit{trivial bundle with fiber} $Z$ if $p_E \colon \Omega \times Z \longrightarrow K$ is the projection onto the first component and $\Omega \times Z$ is equipped with the product topology.
\item Consider a Riemannian manifold $M$. Then the tangent bundle $\mathrm{T}M$ over $M$ is a continuous Banach bundle over $M$.
\item Let $\pi \colon L \longrightarrow K$ be a continuous surjection between compact spaces $L$ and $K$. For each $k \in K$ let $L_k \defeq \pi^{-1}(k)$ be the associated fiber. We define 
\begin{align*}
&E \defeq \bigcupdot_{k \in K} \mathrm{C}(L_k),\\
&p_E \colon E \longrightarrow K, \quad v \in \mathrm{C}(L_k) \mapsto k
\end{align*}
and endow this with the topology generated by the sets
\begin{align*}
W(s,U,\varepsilon) \defeq \{v \in p_E^{-1}(U)\mid \|v - s|_{L_{p(v)}}\|_{\mathrm{C}(L_{p(v)})} < \varepsilon\}
\end{align*}
where $U \subset K$ is open, $s \in \mathrm{C}(L)$ and $\varepsilon > 0$. Then $(E,p_E)$ is a Banach bundle over $K$. Moreover, it is easy to see that $E$ is a continuous Banach bundle if and only if $\pi$ is open. This construction has been used in topological dynamics (see e.g., page 30 of \cite{Knap1967} or Section 5 of \cite{Elli1987}). 
\end{enumerate}
\end{example}
The topology of a Banach bundle is determined by its continuous sections. We make this precise by the following definition and the subsequent lemma.
\begin{definition}\label{defcontsec}
Let $E$ be a Banach bundle over a locally compact space $\Omega$. A continuous mapping $s \colon \Omega \longrightarrow E$ is a \emph{continuous section} of $E$ if $p_E \circ s = \mathrm{id}_\Omega$. We write $\Gamma(E)$ for the space of continuous sections of $E$ and 
\begin{align*}
\Gamma_0(E) &\defeq \{s \in \Gamma(E)\mid \forall \varepsilon >0 \exists K \subset \Omega \textrm{ compact with } \|s(x)\| \leq \varepsilon \forall x \notin K\}
\end{align*}
for the subspace of all \emph{continuous sections vanishing at infinity}.
\end{definition}
\begin{lemma}\label{topology}
Let $E$ be a Banach bundle over a locally compact space $\Omega$. For $v \in E$ the sets
\begin{align*}
V(s,U,\varepsilon) \defeq \{w \in p_E^{-1}(U) \mid \|w - s(p_E(w))\| < \varepsilon\},
\end{align*}
with $s \in \Gamma_0(E)$ satisfying $s(p_E(v)) = v$, $U \subset \Omega$ an open neighborhood of $p(v)$ and $\varepsilon > 0$, form a neighborhood base of $v$ in $E$.
\end{lemma}
\begin{proof}
In the case of a compact base space this follows from Consequences 1.6 (vii) and Theorem 3.2 of \cite{Gier1998} (note that by the proof of Proposition 2.2 of \cite{Gier1998} we may confine ourselves to considering globally defined sections). The general case can readily be reduced to this by considering $\tilde{E}$ (cf. \cref{extendedbundle}).
\end{proof}
In order to define dynamics on Banach bundles we need morphisms between them (cf. page 17 of \cite{DuGi1983}).
\begin{definition}\label{bundlemorph}
Let $\Omega$ be a locally compact space and $\varphi \colon \Omega \longrightarrow \Omega$ a continuous mapping. Consider Banach bundles $E$ and $F$ over $\Omega$. A \emph{(bounded) Banach bundle morphism over $\varphi$} from $E$ to $F$ is a continuous mapping 
\begin{align*}
\Phi \colon E \longrightarrow F
\end{align*}
such that
\begin{enumerate}[(i)]
\item $p_F\circ\Phi = \varphi \circ p_E$, i.e., the diagram
\[
\xymatrix{
E \ar[d]_-{p_E}\ar[r]^{\Phi} &F\ar[d]^-{p_F}\\
\Omega \ar[r]^{\phi} &\Omega
}
\]
commutes,
\item $\Phi|_{E_x} \in \mathscr{L}(E_x,F_{\varphi(x)})$ for each $x \in \Omega$,
\item $\|\Phi\|\defeq \sup_{x \in \Omega} \|\Phi|_{E_x}\|_{\mathscr{L}(E_x,F_{\phi(x)})} < \infty$.
\end{enumerate} 
\end{definition}
Moreover, $\Phi$ is \emph{isometric} if $\Phi|_{E_x}$ is an isometry for each $x \in \Omega$. If $\varphi = \mathrm{id}_{\Omega}$, we simply call a Banach bundle morphism over $\varphi$ a \emph{Banach bundle morphism}.\medskip

\begin{remark}
If $\Omega = K$ is compact, then conditions (i) and (ii) of \cref{bundlemorph} already imply (iii). This can be seen using the same arguments as in the proof of Proposition 1.4 of \cite{DuGi1983}.
\end{remark}
We are interested in dynamical Banach bundles over invertible dynamical systems. Therefore we fix a \emph{topological $G$-dynamical system} $(\Omega;\varphi)$ for the rest of the section, i.e., $\Omega$ is assumed to be a locally compact space and $G$ is a locally compact group acting on $\Omega$ via the continuous mapping
\begin{align*}
\varphi\colon  G\times \Omega \longrightarrow \Omega, \quad (g,x) \mapsto \phi_g(x) = gx.
\end{align*}
Moreover, let $S \subset G$ be a closed subsemigroup of $G$ containing the neutral element $e$, i.e., a closed submonoid of $G$. Important examples of this situation are the cases of $G= \Z$, $S = \N_0$ and $G= \R$, $S= \R_{\geq 0}$.

\begin{definition}\label{defdynamicalbundle}
An \emph{$S$-dynamical Banach bundle over $(\Omega;\varphi)$} is a pair $(E;\Phi)$ of a Banach bundle $E$ over $\Omega$ and a monoid homomorphism
\begin{align*}
\Phi \colon S \longrightarrow E^{E}, \quad g \mapsto \Phi_g,
\end{align*}
such that 
\begin{enumerate}[(i)]
\item the mapping
\begin{align*}
\Phi_g \colon E \longrightarrow E
\end{align*}
is a Banach bundle morphism over $\varphi_g$ for each $g \in S$,
\item $\Phi$ is \emph{jointly continuous}, i.e., the mapping
\begin{align*}
S \times E \longrightarrow E, \quad (g,v) \mapsto \Phi_gv
\end{align*}
is continuous,
\item $\Phi$ is \emph{locally bounded}, i.e., $\sup_{g \in K} \|\Phi_g\| < \infty$ for every compact subset $K \subset S$.
\end{enumerate}
A \emph{morphism} from an $S$-dynamical Banach bundle $(E;\Phi)$ over $(\Omega;\varphi)$ to an $S$-dynamical Banach bundle $(F;\Psi)$ over $(\Omega;\varphi)$ is a Banach bundle morphism $\Theta \colon E \longrightarrow F$ such that the diagram
\[
\xymatrix{
E \ar[d]_-{\Phi_g}\ar[r]^{\Theta} & F\ar[d]^-{\Psi_g}\\
E \ar[r]_{\Theta} & F
}
\]
commutes for each $g \in S$.
\end{definition}

\begin{remark}
The concept of a dynamical Banach bundle is closely related to the notion of cocycles and linear skew-product flows (cf. Definition 6.1 of \cite{ChLa1999}). In fact, if $(E;\Phi)$ is an $S$-dynamical Banach bundle over $(\Omega;\varphi)$, the operators $\Phi_{g,x}\defeq \Phi_g|_{E_x} \in \mathscr{L}(E_x, E_{\varphi_g(x)})$ for $g \in S$ and $x \in K$ satisfy the \emph{cocycle rule}
\begin{align*}
\Phi_{g_1g_2,x} = \Phi_{g_1,\phi_{g_2}(x)}\circ \Phi_{g_2,x}
\end{align*}
for all $g_1,g_2 \in S$ and $x \in \Omega$.
\end{remark}

If $\Omega = K$ is compact, then---once again---a simple adaptation of the arguments of proof of Proposition 1.4 of \cite{DuGi1983} shows that the third condition in \cref{defdynamicalbundle} is superfluous.

\begin{proposition}\label{compactlocallybounded}
Let $\Omega = K$ be compact. Then every monoid homomorphism $\Phi \colon S \longrightarrow E^{E}, \, g \mapsto \Phi_g$ satisfying conditions (i) and (ii) of \cref{defdynamicalbundle} defines an  $S$-dynamical Banach bundle over $(\Omega;\varphi)$.
\end{proposition}
\begin{proof}
Pick $x \in L$ and $g \in S$. Since $\Phi_g 0_x = 0_{\varphi_g(x)}$ we find an open neighborhood $U$ of $x$, $\varepsilon > 0$ and an open neighborhood $V$ of $g$ such that
\begin{align*}
\Phi_h(\{v \in p_E^{-1}(U) \mid \|v\| \leq \varepsilon\}) \subset \{w \in E \mid \|w\| \leq 1\}
\end{align*}
for every $h \in V$. But then $\|\Phi_g|_{E_y}\| \leq \frac{1}{\varepsilon}$ for every $g \in V$ and $y \in U$. Compactness now yields the claim.
\end{proof}

Now we consider dynamics on the Banach bundles of \cref{examples1}.
\begin{example}
\begin{enumerate}[(i)]
\item Assume that $G= \R$, $S= \R_{\geq 0}$, $Z$ is a Banach space and $E= \Omega \times Z$ is the corresponding trivial Banach bundle.\\ If $\{\Phi^t(x) \in \mathscr{L}(Z)\mid x \in \Omega, t \geq 0\}$ is a strongly continuous exponentially bounded cocycle in the sense of Definition 6.1 of \cite{ChLa1999}, then the continuous linear skew-product flow $\Phi_t$ given by 
\begin{align*}
\Phi_t(x,v) \defeq (\varphi_t(x),\Phi^t(x)v)
\end{align*}
for $x \in \Omega$, $v \in Z$ and $t \geq 0$ defines an $\R_{\geq 0}$-dynamical Banach bundle $(E;\Phi)$ over $(\Omega;\varphi)$. Conversely, each $\R_{\geq 0}$-dynamical Banach bundle $(E;\Phi)$ defines a strongly continuous exponentially bounded cocycle by setting
\begin{align*}
\Phi^t(x)v \defeq \mathrm{pr}_2(\Phi_t(x,v))
\end{align*}
for $x \in \Omega$, $v \in Z$ and $t \geq 0$, where $\mathrm{pr}_2\colon \Omega \times Z \longrightarrow Z$ is the projection onto the second component.\\
In particular, evolution families (see Example 6.5 of \cite{ChLa1999} and Section IV.9 of \cite{EN00}) define $\R_{\geq 0}$-dynamical Banach bundles.
\item If $\Omega = M$ is a Riemannian manifold and $\varphi_g \colon M \longrightarrow M$ is differentiable for each $g \in G$, then, by the chain rule, the differentials $\mathrm{D}\varphi_g$ define a $G$-dynamical Banach bundle $(\mathrm{T}M;\mathrm{D}\varphi)$ over $(M;\varphi)$ if $D\varphi$ is locally bounded.
\item Assume that $\Omega = K$ is compact and $\pi \colon (L;\psi) \longrightarrow (K;\varphi)$ is an extension of topological $G$-dynamical systems, i.e., a continuous surjection intertwining the dynamics, and $E$ is defined as in \cref{examples1} (iii). For each $g \in G$ consider
\begin{align*}
\Phi_g\colon E \longrightarrow E, \quad v \in \mathrm{C}(L_k) \mapsto v \circ \psi_{g^{-1}} \in \mathrm{C}(L_{\varphi_g(k)}).
\end{align*}
This defines a $G$-dynamical Banach bundle $(E;\Phi)$ over $(K;\varphi)$.
\end{enumerate}
\end{example}
\subsection{The measurable case}
A measure space $X$ is a triple $(\Omega_X,\Sigma_X,\mu_X)$ consisting of a set $\Omega_X$, a $\sigma$-algebra $\Sigma_X$ of subsets of $\Omega_X$ and a positive $\sigma$-finite measure $\mu_X \colon \Sigma_X \longrightarrow [0,\infty]$. We also assume that our measure spaces are \emph{complete}, i.e., subsets of null sets are measurable.\\
We define Banach bundles over measure spaces as in Section II.4 of \cite{FeDo1988} or Appendix A.3 of \cite{DeRe2000} (see also \cite{Gutm1993b}).
\begin{definition}\label{defmeasurablebundle}
A \emph{(measurable) Banach bundle} over a measure space $X$ is a triple $(E,p_E,\mathcal{M}_E)$ where $E$ is a set, $p_E \colon E \longrightarrow \Omega_X$ is a surjective mapping such that the fiber $E_x \defeq p_E^{-1}(x)$ is a Banach space for each $x \in \Omega_X$ and $\mathcal{M}_E$ is a linear subspace of
\begin{align*}
\mathcal{S}_E \defeq \{s \colon \Omega_X \longrightarrow E\mid p_E \circ s = \mathrm{id}_{\Omega_X}\}
\end{align*}
such that
\begin{enumerate}[(i)]
\item if $f \colon \Omega_X \longrightarrow \K$ is measurable and $s \in \mathcal{M}_E$, then $fs \in \mathcal{M}_E$, where
\begin{align*}
fs \colon s \longrightarrow E, \quad x \mapsto f(x)s(x),
\end{align*}
\item for each $s \in \mathcal{M}_E$ the mapping
\begin{align*}
|s| \colon \Omega_X \longrightarrow \R_{\geq 0}, \quad x \mapsto \|s(x)\|_{E_x}
\end{align*}
is measurable,
\item if $(s_n)_{n \in \N}$ is a sequence in $\mathcal{M}_E$ converging almost everywhere to $s \in \mathcal{S}_E$, then $s \in \mathcal{M}_E$.
\end{enumerate}
Elements $s \in \mathcal{S}_E$ are called \emph{sections} and elements $s \in \mathcal{M}_E$ are called \emph{measurable sections}.\\
The bundle is \emph{separable} if, in addition, 
\begin{enumerate}[(i)]
\setcounter{enumi}{3}
\item there is a sequence $(s_n)_{n \in \N}$ in $\mathcal{M}_E$ such that $\lin \{s_n(x)\mid n \in \N\}$ is dense in $E_x$ for almost every $x \in \Omega_X$.
\end{enumerate}
\end{definition}
We mostly just write $E$ for a measurable Banach bundle $(E,p_E,\mathcal{M}_E)$.
\begin{remark}\label{generatedmeasbundle}
Let $X$ be a measure space and $(E,p_E)$ a pair of a set $E$ and a surjective mapping $p_E \colon E \longrightarrow \Omega_X$ such that the fiber $E_x \defeq p_E^{-1}(x)$ is a Banach space for each $x \in \Omega_X$. Then by Section II.4.2 of \cite{FeDo1988} every linear subspace $\mathcal{M}_E$ of $\mathcal{S}_E$ satisfying condition (iii) of \cref{defmeasurablebundle} \emph{generates} a measurable Banach bundle, i.e., there is a smallest linear subspace $\tilde{\mathcal{M}}_E$ of $\mathcal{S}_E$ containing $\mathcal{M}_E$ such that $(E,p_E, \tilde{\mathcal{M}}_E)$ is a measurable Banach bundle. Moreover, $\tilde{\mathcal{M}}_E$ consists precisely of all almost everywhere limits of sequences in $\lin\{\mathbbm{1}_As\mid A \in \Sigma_X, s \in \mathcal{M}_E\}$.
\end{remark}
We briefly list some examples for measurable Banach bundles and refer to Appendix A.3 of \cite{DeRe2000} for additional examples.
\begin{example}\label{examplesmeasurable}
\begin{enumerate}[(i)]
\item Let $X$ be a measure space and $Z$ a Banach space. Consider $E \defeq \Omega_X \times Z$ with the projection $p_E$ onto the first component. The space of sections $\mathcal{S}_E$ can be identified with the space of all functions from $\Omega_X$ to $Z$. The set of all strongly measurable functions (see Section 1.3.5 of \cite{HiPh1957}) then defines a subset $\mathcal{M}_E$ of $\mathcal{S}_E$ which turns $E$ into a measurable Banach bundle called the \emph{trivial Banach bundle with fiber} $Z$. This coincides with the measurable Banach bundle generated by the constant sections (see Section II.5.1 of \cite{FeDo1988}).
\item Let $E$ be a topological Banach bundle over a locally compact space $\Omega$, $\mu$ be a $\sigma$-finite regular Borel measure on $\Omega$ and $\mathcal{B}(\Omega)$ the Borel $\sigma$-algebra of $\Omega$. Then the space $\Gamma(E)$ (see \cref{defcontsec}) generates a measurable Banach bundle $E_\mu$ over the completion of the measure space $(\Omega,\mathcal{B}(\Omega),\mu)$. See Section II.15 of \cite{FeDo1988} for a more explicit description of the measurable sections of a continuous Banach bundle.
\end{enumerate}
\end{example}
Before introducing dynamics on measurable Banach bundles, we first define morphisms of measure spaces. A \emph{premorphism} $\varphi \colon X \longrightarrow Y$ between measure spaces $X$ and $Y$ is a measurable and measure-preserving mapping $\varphi \colon \Omega_X \longrightarrow \Omega_Y$. Setting $\varphi \sim \psi$ if $\varphi(x) = \psi(x)$ for almost every $x \in \Omega_X$ defines an equivalence relation on the set of premorphisms from $X$ to $Y$. The equivalence classes with respect to this equivalence relation are then the \emph{morphisms} from $X$ to $Y$. As usual, given a morphism we will implicitly choose a representative of it but also denote it by $\varphi$ when there is no room for confusion.\\
We now define morphisms of measurable Banach bundles in a similar manner.
\begin{definition}\label{premorphdef}
Let $\varphi \colon X \longrightarrow X$ be a morphism on a measure space $X$. Consider Banach bundles $E$ and $F$ over $X$. A \emph{premorphism} $\Phi$ from $E$ to $F$ over $\varphi$ is a mapping $\Phi \colon E \longrightarrow F$ such that
\begin{enumerate}[(i)]
\item $\Phi \circ \mathcal{M}_E \subset \mathcal{M}_F \circ \varphi$,
\item $p_F \circ \Phi = \phi \circ p_E$ almost everywhere,
\item $\Phi|_{E_x} \in \mathscr{L}(E_x,F_{\varphi(x)})$ for almost every $x \in \Omega_X$,
\item $\|\Phi\| \defeq \esssup_{x \in \Omega_X} \|\Phi|_{E_x}\| < \infty$.
\end{enumerate}
\end{definition}
Again, we want to identify premorphisms which agree up to a null set. Set
\begin{align*}
\mathrm{Premor}_\varphi(E,F) &\defeq \{\Phi \colon E \longrightarrow F \textrm{ premorphism over } \varphi\},\\
\mathcal{N}_\varphi(E,F) &\defeq \{\Phi \in \mathrm{Premor}_\varphi(E,F)\mid \Phi = 0 \textrm{ almost everywhere}\},
\end{align*}
and $\mathrm{Mor}_\varphi(E,F) \defeq \mathrm{Premor}_\varphi(E,F)/\mathcal{N}_\varphi(E,F)$ for measurable Banach bundles $E$ and $F$ as above.\\
An equivalence class $[\Phi] \in \mathrm{Mor}_\varphi(E,F)$ is called a \emph{morphism of measurable Banach bundles over $\varphi$}. It is \emph{isometric} if $\Phi|_{E_x}$ is isometric for almost every $x \in \Omega_X$. If $\varphi = \mathrm{id}_X$, we call a morphism over $\varphi$ simply a \emph{morphism of measurable Banach bundles}.\\
As above, we will implicitly choose representatives of morphisms whenever necessary and denote them with the same symbol.\medskip

Now we introduce dynamical measurable Banach bundles.
For the rest of this section let $G$ be a group with neutral element $e \in G$. We fix a \emph{measure-preserving $G$-dynamical system} $(X;\varphi)$, i.e., a measure space $X$ together with a group homomorphism 
\begin{align*}
\varphi\colon G \longrightarrow \mathrm{Aut}(X), \quad g \mapsto \varphi_g,
\end{align*}
where $\mathrm{Aut}(X)$ is the set of automorphisms of $X$. Also fix a submonoid $S \subset G$, i.e., a subsemigroup containing $1 \in G$.
\begin{definition}
An \emph{$S$-dynamical Banach bundle over $(X;\varphi)$} is a pair $(E;\Phi)$ of a measurable Banach bundle $E$ over $X$ and a family $(\Phi_g)_{g \in S}$ with $\Phi_g \colon E \rightarrow E$ a morphism over $\varphi_g$ for $g \in S$ such that
	\begin{itemize}
		\item $\Phi_g \circ \Phi_h = \Phi_{gh}$ for all $g,h \in S$,
		\item $\Phi_1 = \mathrm{Id}_E$.
	\end{itemize}
We call $(E;\Phi)$ \emph{separable} if $E$ is separable.\smallskip

A \emph{morphism} between measurable Banach bundles $(E;\Phi)$ and $(F;\Psi)$ over $(X;\varphi)$ is a morphism $\Theta\colon E \longrightarrow F$ of Banach bundles such that the diagram
\[
\xymatrix{
E \ar[d]_-{\Phi_g}\ar[r]^{\Theta} & F\ar[d]^-{\Psi_g}\\
E \ar[r]_{\Theta} & F
}
\]
commutes for each $g \in S$.
\end{definition}
\begin{example}
\begin{enumerate}[(i)]
\item Let $E$ be the trivial bundle with fiber $Z$ (see \cref{examplesmeasurable} (i)). Then the $S$-dynamical Banach bundles correspond to \emph{measurable cocycles}, i.e., mappings $\Phi \colon S \times X \longrightarrow \mathscr{L}(Z)$ such that
\begin{itemize}
	\item $\Phi(gh,x) = \Phi(g,\varphi_h(x)) \circ \Phi(h,x)$ for almost every $x \in X$ for all $g,h \in S$,
	\item $\Phi(1,x) = \mathrm{Id}_Z$ for almost every $x \in X$,
	\item $X \longrightarrow Z, \, x \mapsto \Phi(g,x)v$ is strongly measurable for all $g \in S$ and $v \in Z$,
	\item $\esssup_{x \in \Omega_X} \|\Phi(g,x)\| < \infty$ for every $g \in S$.
\end{itemize}
\item Let $(E;\Phi)$ be a topological $S$-dynamical Banach bundle over a topological $G$-dynamical system $(\Omega;\varphi)$ (with $G$ and $S$ discrete) and let $\mu$ be a $\sigma$-finite regular Borel measure on $\Omega$. Moreover, let $E_\mu$ be the induced measurable Banach bundle of \cref{examplesmeasurable} (ii). Then $(E_\mu;\Phi)$ is an $S$-dynamical measurable Banach bundle over the mea\-sure-preserving $G$-dynamical system induced by $(\Omega;\varphi)$.
\end{enumerate}
\end{example}

\section{Dynamical Banach modules}
In the previous sections we have defined dynamics on topological and measurable Banach bundles. We now consider  Banach modules as the operator theoretic counterparts. 
First we recall the following definition from Section 2 of \cite{DuGi1983}. 
\begin{definition}
Let $A$ be a commutative Banach algebra. A Banach space $\Gamma$ which is also an $A$-module is a \emph{Banach module over $A$} if $\|fs\| \leq \|f\|\|s\|$ for all $f \in A$ and $s \in \Gamma$.\smallskip

A \emph{homomorphism} from a Banach module $\Gamma$ over $A$ to a Banach module $\Lambda$ over $A$ is a bounded operator $T \in \mathscr{L}(\Gamma,\Lambda)$ which is also an $A$-module homomorphism. It is \emph{isometric} if $T$ is an isometry.
\end{definition}
In the following we always assume that Banach modules $\Gamma$ over a commutative Banach algebra $A$ are \emph{non-degenerate} (see \cite{Para2008}) in the sense that
\begin{align*}
\Gamma = \overline{\lin}\,\{fs\mid f \in A, s \in \Gamma\}.
\end{align*}
Note that if $A$ is a commutative C*-algebra (if $\K = \C)$ or its self-adjoint part (if $\K = \R$) and $(e_i)_{i \in I}$ is an approximate unit (see Section 1.8 of \cite{Dixm1977}), then this is the case if and only if $\lim_i e_i s = s$ for each $s \in \Gamma$. In particular, if $A$ has a unit, then the module is unitary.\smallskip

We now discuss Banach modules associated with Banach bundles. 
\begin{example}\label{examplesections}
Let $E$ be a topological Banach bundle over a locally compact space $\Omega$.
Then $\Gamma_0(E)$ (see \cref{defcontsec}) is a Banach module over $\mathrm{C}_0(\Omega)$ if equipped with the operation
\begin{align*}
\mathrm{C}_0(\Omega) \times \Gamma_0(E) \longrightarrow \Gamma_0(E), \quad (f,s) \mapsto [x \mapsto f(x)s(x)]
\end{align*}
and the norm $\|\cdot\|$ defined by $\|s\| \defeq \sup_{x \in \Omega} \|s(x)\|$ for $s \in \Gamma_0(E)$.
\end{example}
\begin{remark}\label{moduleextension}
Let $\Omega$ be a locally compact space and $E$ a Banach bundle over $\Omega$. If $K$ is the one-point compactification of $\Omega$ and $\tilde{E}$ the extended bundle of $E$ (see \cref{extendedbundle}), then 
\begin{align*}
	\Gamma(\tilde{E}) \rightarrow \Gamma_0(E),\quad s \mapsto s|_\Omega
\end{align*}
is an isometric isomorphism of Banach spaces. In particular, we can consider $\Gamma_0(E)$ as a Banach module over $\mathrm{C}(K)$.
\end{remark}
\begin{example}\label{measurablesection}
For a measurable Banach bundle $E$ over a measure space $X$ we define
\begin{align*}
\mathcal{N}_E &\defeq \{s \in \mathcal{M}_E\mid s = 0 \textrm{ almost everywhere}\},\\
\Gamma^1(E) &\defeq \left\{s \in \mathcal{M}_E \mid |s| \textrm{ is integrable} \right\}/\mathcal{N}_E,\\
\Gamma^\infty(E) &\defeq \left\{s \in \mathcal{M}_E \mid |s| \textrm{ is essentially bounded}\right\}/\mathcal{N}_E.
\end{align*}
With the natural norms and operations the spaces $\Gamma^1(E)$ and $\Gamma^\infty(E)$ are Banach modules over $\mathrm{L}^\infty(X)$.
\end{example}

In order to define dynamical Banach modules we now proceed as above and define first \enquote{morphisms over morphisms}.

\begin{definition}
Let $A$ be a commutative Banach algebra and $T \in \mathscr{L}(A)$ an algebra homomorphism. Moreover, let $\Gamma$ and $\Lambda$ be Banach modules over $A$. Then $\mathcal{T} \in \mathscr{L}(\Gamma,\Lambda)$ is a \emph{$T$-homomorphism} if
\begin{align*}
\mathcal{T}(fs) = Tf\cdot \mathcal{T}s \textrm{ for all } f \in A \textrm{ and } s \in \Gamma.
\end{align*}
\end{definition}
\begin{example}\label{examplemorphover}
\begin{enumerate}[(i)]
\item Let $\varphi \colon \Omega \longrightarrow \Omega$ be a homeomorphism of a locally compact space $\Omega$. Then the \emph{Koopman operator} $T_\varphi \in \mathscr{L}(\mathrm{C}_0(\Omega))$ defined by $T_\varphi f \defeq f \circ \varphi^{-1}$ for $f \in \mathrm{C}_0(\Omega)$ is an algebra automorphism.\\
If $E$ and $F$ are Banach bundles over $\Omega$ and $\Phi \colon E \longrightarrow F$ is a Banach bundle morphism over $\varphi$, the \emph{weighted Koopman operator} $\mathcal{T}_\Phi \in \mathscr{L}(\Gamma_0(E),\Gamma_0(F))$ given by $\mathcal{T}_\Phi s \defeq \Phi \circ s\circ \varphi^{-1}$ for $s \in \Gamma_0(E)$ is a $T_\varphi$-homomorphism.
\item Let $\varphi \colon X \longrightarrow X$ be an automorphism of a measure space $X$. Then the \emph{Koopman operator} $T_\varphi \in \mathscr{L}(\mathrm{L}^\infty(X))$ defined by $T_\varphi f \defeq f \circ \varphi^{-1}$ for $f \in \mathrm{L}^\infty(X)$ is an algebra automorphism.\\
If $E$ and $F$ are Banach bundles over $X$ and $\Phi \colon E \longrightarrow F$ is a Banach bundle morphism over $\varphi$, the \emph{weighted Koopman operator} $\mathcal{T}_\Phi \in \mathscr{L}(\Gamma^1(E),\Gamma^1(F))$ given by $\mathcal{T}_\Phi s \defeq \Phi \circ s\circ \varphi^{-1}$ for $s \in \Gamma^1(E)$ is a $T_\varphi$-homomorphism. Similarly, $\Phi$ induces an operator $\mathcal{T}_\Phi \in \mathscr{L}(\Gamma^\infty(E),\Gamma^\infty(F))$.
\end{enumerate} 
\end{example}
Before introducing the concept of dynamical Banach modules we prove a different characterization of $T$-homomorphisms as some sort of \enquote{locality preserving operators}. We start with the following definition.
\begin{definition}
Let $A$ be a commutative Banach algebra and $\Gamma$ a Banach module over $A$. For $s \in \Gamma$ we call the closed ideal
	\begin{align*}
		I_s\defeq \{f \in A \mid fs = 0\}
	\end{align*}
the \emph{supporting ideal of} $s$ \emph{in} $A$.
\end{definition}
If $A = \mathrm{C}_0(\Omega)$ for some locally compact space $\Omega$, then there is a correspondence between the concept of supporting ideals and the following notion of support (see Definition 9.3 of \cite{AbArKi1992}).
\begin{definition}\label{defsupport}
Let $\Omega$ be a locally compact space and $\Gamma$ a Banach module over $\mathrm{C}_0(\Omega)$. For $s \in \Gamma$ we call
\begin{align*}
\supp(s) \defeq \{x \in \Omega \mid \textrm{each } f \in \mathrm{C}_0(\Omega) \textrm{ with } f(x) \neq 0 \textrm{ satisfies } fs \neq 0\} \subset \Omega
\end{align*}
the \emph{support of $s$ in $\Omega$}.
\end{definition}
\begin{lemma}\label{descriptionsupport}
Let $\Omega$ be a locally compact space and $\Gamma$ a Banach module over $\mathrm{C}_0(\Omega)$. Then 
\begin{align*}
	I_s = \{f \in \mathrm{C}_0(\Omega)\mid f|_{\supp(s)} = 0\}.
\end{align*}
for every $s \in \Gamma$.
\end{lemma}
\begin{proof}
Let $s \in \Gamma$. Since $I_s$ is a closed ideal in $\mathrm{C}_0(\Omega)$, we find a unique closed subset $M$ such that $f|_M = 0$ if and only if $f \in I_s$. It is clear that $\supp(s) \subset M$.
On the other hand, if $x \in \Omega \setminus \supp(s)$, we find $f \in \mathrm{C}_0(\Omega)$ with $f(x) \neq 0$ but $fs = 0$. Then $f|_M = 0$ which shows $x \notin M$.
\end{proof}
The following is a first characterization of $T$-homomorphisms extending Theorem 9.5 of \cite{AbArKi1992}.
\begin{theorem}\label{supportlemma}
Let $\varphi \colon \Omega \longrightarrow \Omega$ be a homeomorphism of a locally compact space $\Omega$ and $\Gamma$ and $\Lambda$ Banach modules over $\mathrm{C}_0(\Omega)$. For $\mathcal{T} \in \mathscr{L}(\Gamma,\Lambda)$ the following assertions are equivalent.
\begin{enumerate}[(a)]
	\item $\mathcal{T}$ is a $T_\varphi$-homomorphism.
	\item $T_\varphi I_s \subset I_{\mathcal{T}s}$ for every $s \in \Gamma$.
	\item $\supp(\mathcal{T}s) \subset \varphi(\supp(s))$ for each $s \in \Gamma$.
\end{enumerate}
\end{theorem}
For the proof we need the following lemma.
\begin{lemma}\label{extendedmodule}
Let $\Omega$ be a locally compact space and $\Gamma$ be a Banach module over $\mathrm{C}_0(\Omega)$. Let $K = \Omega \cupdot \{\infty\}$ be the one-point compactification of $\Omega$. The mapping 
\begin{align*}
\mathrm{C}(K) \times \Gamma \longrightarrow \Gamma, \quad (f,s) \mapsto (f-f(\infty)\mathbbm{1})|_\Omega s + f(\infty)s
\end{align*}
turns $\Gamma$ into a (unitary) Banach module over $\mathrm{C}(K)$. 
\end{lemma}
\begin{proof}
It is easy to check that the mapping above actually turns $\Gamma$ into a unitary module over $\mathrm{C}(K)$. Choose an approximate unit $(e_i)_{i \in I}$ for $\mathrm{C}_0(\Omega)$. Now take $f \in \mathrm{C}(K)$ and $s \in \Gamma$ and observe that
\begin{align*}
\|fs\| &= \lim_{i} \|(f-f(\infty)\mathbbm{1})|_\Omega e_i s + f(\infty)e_i s\|\\
&= \lim_{i} \|(fe_i)s\| \leq \limsup_{i} \|e_i f\| \|s\|\\
&\leq \|f\| \|s\|.
\end{align*}
This shows $\|fs\| \leq \|f\| \|s\|$ and therefore $\Gamma$ is a Banach module over $\mathrm{C}(K)$.
\end{proof}
\begin{proof}[of \cref{supportlemma}]
The equivalence of (b) and (c) is obvious by Tietze's theorem while the equivalence of (a) and (c) follows from Theorem 9.5 of \cite{AbArKi1992} if $K= \Omega$ is compact and $\varphi = \mathrm{id}_K$\footnote{Note that even though the authors work in the complex setting, their proof also works in the real case.}.\\
Now take $\Omega$ non-compact but still assume $\varphi= \mathrm{id}_\Omega$. We consider the one-point compactification $K$ of $\Omega$ and the module structure of $\Gamma$ over $\mathrm{C}(K)$ (see \cref{extendedmodule}). For $s \in \Gamma$ we denote the support of $s$ with respect to this module structure by $\supp_K(s)$. It is easy to see that 
\begin{align*}
\overline{\supp(s)}^K \subset \supp_K(s)  \subset \supp(s) \cup \{\infty\}.
\end{align*}
Let $(e_i)_{i \in I}$ be an approximate unit for $\mathrm{C}_0(\Omega)$.
It is easy to see that $\infty \notin \supp_K(s)$ if and only if there is $g \in \mathrm{C}_0(\Omega)$ with $gs = s$. But this is the case if and only if there is $i_0\in A$ with $(e_i g - e_i)s = 0$, i.e., $(e_i g - e_i)|_{\supp(s)} = 0$ for every $i \geq i_0$. Therefore, the result for non-compact $\Omega$ can be reduced to the compact case.\smallskip

Finally let $\varphi \colon \Omega \longrightarrow \Omega$ be an arbitrary homeomorphism of a locally compact space $\Omega$. Consider the module $\Lambda_{T_\phi}$ which is the space $\Lambda$ equipped with the new operation $f \cdot_{T_\phi} s \defeq T_\phi f \cdot s$ for $f \in \mathrm{C}_0(\Omega)$ and $s \in \Lambda$. Then $\mathcal{T}\in \mathscr{L}(\Gamma,\Lambda)$ is a $T_\varphi$-homomorphism if and only if $\mathcal{T} \in \mathscr{L}(\Gamma,\Lambda_{T_\phi})$ is a homomorphism of Banach modules. By the above, this is the case if and only if
\begin{align*}
\{x \in \Omega \mid \textrm{each } f \in \mathrm{C}_0(\Omega) \textrm{ with } f(x) \neq 0 \textrm{ satisfies } T_\varphi f \cdot \mathcal{T}s \neq 0\} \subset \supp(s),
\end{align*}
i.e., $\supp(\mathcal{T}s)\subset \varphi(\supp(s))$ for each $s \in \Gamma$.
\end{proof}
We now introduce dynamical Banach modules. Fix a pair $(A;T)$ of a commutative Banach algebra $A$ and a strongly continuous group representation $T \colon G \longrightarrow \mathscr{L}(A)$ of a locally compact group $G$ as algebra automorphisms of $A$. Moreover, let $S \subset G$ be a fixed closed submonoid.
\begin{definition}
An \emph{$S$-dynamical Banach module over $(A;T)$} is a pair $(\Gamma; \mathcal{T})$ consisting of a Banach module $\Gamma$ over $A$ and a monoid homomorphism $\mathcal{T}\colon S \longrightarrow \mathscr{L}(\Gamma)$ such that
\begin{enumerate}[(i)]
\item $\mathcal{T}(g)\in \mathscr{L}(\Gamma)$ is a $T(g)$-homomorphism for each $g \in S$,
\item $\mathcal{T}$ is \emph{strongly continuous}, i.e., 
	\begin{align*}
		S \longrightarrow \Gamma,\quad g \mapsto \mathcal{T}(g)s
	\end{align*}
	is continuous for every $s \in \Gamma$.
\end{enumerate}
A \emph{homomorphism} from an $S$-dynamical Banach module $(\Gamma;\mathcal{T})$ over $(A;T)$ to an $S$-dynamical Banach module $(\Lambda;\mathcal{S})$ over $(A;T)$ is a homomorphism $V \in \mathscr{L}(\Gamma,\Lambda)$ of Banach modules over $A$ such that the diagram
\[
\xymatrix{
\Gamma \ar[d]_-{\mathcal{T}(g)}\ar[r]^{V} & \Lambda\ar[d]^-{\mathcal{S}(g)}\\
\Gamma \ar[r]_{V} & \Lambda
}
\]
commutes for each $g \in S$.
\end{definition}

Starting with the topological case, we now show that dynamical Banach bundles induce dynamical Banach modules.
\begin{example}\label{exampleThom}
Consider an $S$-dynamical Banach bundle $(E;\Phi)$ over a topological $G$-dynamical system $(\Omega;\varphi)$. For each $g \in G$ the Koopman operator $T_\varphi(g) \defeq T_{\varphi_g}$ is an auotmorphism of $\mathrm{C}_0(\Omega)$ (see \cref{examplemorphover} (i)) and $g \mapsto T_\varphi(g)$ defines a representation of $G$ as operators on $\mathrm{C}_0(\Omega)$, called the \emph{Koopman representation} which is strongly continuous (this is probably well-known, but also a special case of \cref{inducedsystem1} below). \\
By setting $\mathcal{T}_{\Phi}(g) \defeq \mathcal{T}_{\Phi_g}$ for each $g \in S$ we obtain a $T_\varphi(g)$-homomorphism $\mathcal{T}_\Phi(g) \in \mathscr{L}(\Gamma_0(E))$ for each $g \in S$ (see \cref{examplemorphover}). We call the monoid representation $\mathcal{T}_\Phi$ the \emph{weighted Koopman representation} of $(E;\Phi)$.
\end{example}
\begin{proposition}\label{inducedsystem1}
Let $(\Omega;\varphi)$ be a topological $G$-dynamical system, $A= \mathrm{C}_0(\Omega)$ and $T=T_\varphi$ the Koopman representation of $(\Omega;\varphi)$.
\begin{enumerate}[(i)]
\item If $(E;\Phi)$ is an $S$-dynamical Banach bundle over $(\Omega;\varphi)$, then the weighted Koopman representation $\mathcal{T}_\Phi$ defines an $S$-dynamical Banach module over $(\mathrm{C}_0(\Omega);T_\varphi)$. 
\item For a morphism $\Theta \colon (E;\Phi) \longrightarrow (F;\Psi)$ of $S$-dynamical Banach bundles over $(\Omega;\varphi)$ the operator $V_\Theta \in  \mathscr{L}(\Gamma_0(E),\Gamma_0(F))$ defined by
\begin{align*}
V_\Theta s \defeq \Theta \circ s \quad \textrm{ for } s\in \Gamma_0(E)
\end{align*}
is a homomorphism $V_\Theta \in  \mathscr{L}(\Gamma_0(E),\Gamma_0(F))$ between the $S$-dynamical Banach modules $(\Gamma_0(E);\mathcal{T}_\Phi)$ and $(\Gamma_0(F);\mathcal{T}_\Psi)$.
\end{enumerate}
\end{proposition}
For the proof we need the following lemma.
\begin{lemma}\label{locallycomp}
Let $(E;\Phi)$ be an $S$-dynamical Banach bundle over $(\Omega;\varphi)$. Let $K\defeq \Omega \cupdot \{\infty\}$ be the one-point compactification of $\Omega$ and $\tilde{E}$ the extended Banach bundle of \cref{extendedbundle}. Then the following assertions hold.
\begin{enumerate}[(i)]
\item The mapping
\begin{align*}
\tilde{\varphi}\colon G \times K \longrightarrow K,\quad (g,x) \mapsto \begin{cases} \infty& x = \infty,\\
\varphi(g,x) & x \neq \infty,
\end{cases}
\end{align*}
is continuous.
\item Setting
\begin{align*}
\tilde{\Phi}\colon S \times \tilde{E}\longrightarrow \tilde{E},\quad (g,v) \mapsto \begin{cases} 0 & v \in E_\infty,\\
\Phi_g v & v \in E,
\end{cases}
\end{align*}
defines an $S$-dynamical Banach bundle $(\tilde{E};\tilde{\Phi})$ over $(K;\tilde{\varphi})$.
\end{enumerate}
\end{lemma}
\begin{proof}
If $g \in G$ and $L$ is a compact subset of $\Omega$, we choose a compact neighborhood $V$ of $g$ and set $U \defeq (V^{-1} \cdot L)^c$. Then $U$ is cocompact with $hy \notin L$ for all $h \in V$ and $y \in U$. This shows (i).\smallskip

Now take $\varepsilon >0$ and assume that $g \in S$. Since $\Phi$ is locally bounded, we find a $\delta > 0$ with $\|\Phi_h\|<\frac{1}{\delta}$ for every $h \in V \cap S$.
For $v \in E$ with $\|v\| < \delta\varepsilon$ and $p_E(v) \in U$ and $h \in V \cap S$ we then have $p_{\tilde{E}}(\Phi_h v) \notin L$ and $\|\Phi_h v\| < \varepsilon$, i.e., $\Phi_h v \in U(L,\varepsilon)$ in the notation of \cref{extendedbundle}. This shows that $\tilde{\Phi}$ is jointly continuous.
\end{proof}

\begin{proof}[of \cref{inducedsystem1}]
We first prove continuity of the weighted Koopman representation in the case of a compact space $\Omega = K$. Fix $s \in \Gamma(E)$ and let $g \in S$ and $\varepsilon > 0$. For $x \in K$ the set
\begin{align*}
V \defeq V(\Phi_g \circ s \circ \varphi_{g^{-1}},K, \varepsilon) \defeq \{v \in E \mid \|v - \Phi_g s(g^{-1}(p_E(v)))\| < \varepsilon\}
\end{align*}
is a neighborhood of $\Phi_g s(g^{-1}x)$. Since the mapping
\begin{align*}
S \times K \longrightarrow E, \quad (h,y) \mapsto \Phi_h s(y)
\end{align*} 
is continuous as a composition of the continuous mappings
\begin{align*}
&S \times K \longrightarrow S \times E, \quad (h,y) \mapsto (h,s(y)),\\
&S \times E \longrightarrow E, \quad (h,v) \mapsto \Phi_h v,
\end{align*}
we find a neighborhood $O\subset S$ of $g$ and a neighborhood $U\subset K$ of $g^{-1}x$ such that $\Phi_h s(y) \in V$ for every $h \in O$ and $y \in U$, i.e., 
\begin{align*}
\|\Phi_h s(y) - \Phi_g s(g^{-1}hy))\| < \varepsilon.
\end{align*}
By compactness of $K$ we thus find a neighborhood $W \subset S$ of $g$ with
\begin{align*}
\sup_{y \in K}\|\Phi_h s(y) - \Phi_g s(g^{-1}hy))\| < \varepsilon
\end{align*}
for all $h \in W$. But then
\begin{align*}
\sup_{y \in K} \|\Phi_h s(h^{-1}y) - \Phi_g s(g^{-1}y))\| = \sup_{y \in K}\|\Phi_h s(y) - \Phi_g s(g^{-1}hy))\| < \varepsilon
\end{align*}
for each $h \in W$.\\
The general case of (i) now follows from \cref{locallycomp} and \cref{moduleextension} and part (ii) is obvious.
\end{proof}

\begin{example}\label{examplemeasurable}
Let $G$ carry the discrete topology, $(X;\varphi)$ be a measure-preserving $G$-dynamical system, $A= \mathrm{L}^\infty(X)$ and $T=T_\varphi$ the induced \emph{Koopman representation} on $\mathrm{L}^\infty(X)$, i.e., $T_\varphi(g) \defeq T_{\varphi_g}$ for every $g \in G$.\\
Then every $S$-dynamical Banach bundle $(E;\Phi)$ over $(X;\varphi)$ induces a \emph{weighted Koopman representation} $\mathcal{T}_\Phi$ on $\Gamma^1(E)$ via $T_{\Phi}(g) \defeq T_{\Phi_g}$ for $g \in S$ which defines an $S$-dynamical Banach module over $(\mathrm{L}^\infty(X);T_\varphi)$. Moreover, if $\Theta \colon (E;\Phi) \longrightarrow (F;\Psi)$ is a morphism of $S$-dynamical Banach bundles over $(X;\varphi)$, then $V_\Theta s \defeq \Theta \circ s$ for $s \in \Gamma^1(E)$ defines a homomorphism from $(\Gamma^1(E);\mathcal{T}_\Phi)$ to $(\Gamma^1(F);\mathcal{T}_\Psi)$. 
\end{example}

\section{AM- and AL-modules}
We have seen that topological and measurable Banach bundles induce dynamical Banach modules and that these assignments are functorial. We now describe the essential ranges of these functors.\\
For this we recall a connection between Banach modules and Banach lattices, observed by Kaijser in Proposition 2.1 of \cite{Kaij1978} and Abramovich, Arenson and Kitover in Lemma 4.6 of \cite{AbArKi1992} in the compact case. We give a new proof for the locally compact case based on Lemma 1 of \cite{Cunn1967} and also provide more details on the lattice structure. Here and in the following we write $E_+$ for the positive cone of a Banach lattice $E$. 
\begin{proposition}\label{modulelattice}
If $\Omega$ is a locally compact space, $\Gamma$ a Banach module over $\mathrm{C}_0(\Omega)$ and $s \in \Gamma$, then the submodule $\Gamma_s \defeq \overline{\mathrm{C}_0(\Omega)\cdot s}$ is a Banach lattice with positive cone $\overline{\mathrm{C}_0(\Omega)_+\cdot s}$. Moreover, we obtain the following for $f, g \in \mathrm{C}_0(\Omega,\R)$ and $h \in \mathrm{C}_0(\Omega)$,
\begin{enumerate}[(i)]
	\item $fs \leq gs$ if and only if $f|_{\supp(s)}\leq g|_{\supp(s)}$,
	\item $(fs \vee gs) = (f\vee g)s$,
	\item $(fs \wedge gs) = (f\wedge g)s$,
	\item $|hs| = |h|s$.
\end{enumerate}
If $\K= \C$, then $\Gamma_s$ is the complexification of the real Banach lattice $\overline{\mathrm{C}_0(\Omega,\R)s}$.
\end{proposition}
\begin{proof}
Take $f, g \in \mathrm{C}_0(\Omega)$ with $|g| \leq |f|$. We show that $\|fs\| \leq \|gs\|$. Set $N\defeq g^{-1}(\{0\})$ and choose an approximate unit $(e_i)_{i \in I}$ for $I_N \defeq \{h \in \mathrm{C}_0(\Omega)\mid h|_N = 0\}$ such that $e_i$ has compact support for every $i \in I$. Also define $h_i \in \mathrm{C}_0(\Omega)$ for $i \in I$ by 
\begin{align*}
	h_i(x) \defeq \begin{cases}
		e_i(x)\frac{g(x)}{f(x)},& x \notin N,\\
		0, & x \in N.
	\end{cases}
\end{align*}
Then $|h_i(x)| \leq 1$ for every $x \in \Omega$ and therefore 
\begin{align}
\|gs\| = \lim_{i} \|e_i gs\| = \lim_{i} \|h_i fs\| \leq \limsup_{i} \|h_i\| \|fs\| \leq \|fs\|.
\end{align}
We set $|fs| \defeq |f|s$ for $f \in \mathrm{C}_0(\Omega)$. By the above we obtain for $f,g \in \mathrm{C}_0(\Omega)$ 
	\begin{align}
		\||f|s - |g|s\| = \|||f| - |g||s\| \leq \||f-g|s\| = \|(f-g)s\| = \|fs - gs\|.
	\end{align}
This implies that $|\cdot| \colon \mathrm{C}_0(\Omega)s \longrightarrow \mathrm{C}_0(\Omega)s$ has a unique extension to a continuous map $|\cdot| \colon \Gamma_s \longrightarrow \Gamma_s$. The only non-trivial part in showing that this defines a modulus in the sense of Definition 1.1 of \cite{MiWo1974} is to check that the linear hull of the image $|\Gamma_s| = \overline{\mathrm{C}_0(\Omega)s}$ is the whole space $\Gamma_s$. However, if $t = \lim_{n \rightarrow \infty} f_ns \in \Gamma_s$, then---using (1) and (2) as well as the formulas for the positive and negative parts of functions (see Corollary 1 of Proposition II.1.4 of \cite{Scha1974})---it is standard to check that $((\mathrm{Re}\, f_n)_+s)_{n \in \N}$,  $((\mathrm{Re}\, f_n)_-s)_{n \in \N}$,  $((\mathrm{Im}\, f_n)_+s)_{n \in \N}$ and  $((\mathrm{Im}\, f_n)_-s)_{n \in \N}$ are Cauchy sequences and therefore converge in $\overline{\mathrm{C}_0(\Omega)_+s}$. This implies that $t$ can be written as a linear combination of elements of $\overline{\mathrm{C}_0(\Omega)_+s}$. Moreover, this shows $\overline{\mathrm{C}_0(\Omega,\R)s} = \overline{\mathrm{C}_0(\Omega)_+s} - \overline{\mathrm{C}_0(\Omega)_+s}$.

By Proposition 1.3 of \cite{MiWo1974}, we obtain that $\overline{\mathrm{C}_0(\Omega)_+s}$ is a cone and defines a partial order on $\overline{\mathrm{C}_0(\Omega,\R)s}$.  Moreover, $\|hs\| = \||hs|\|$ for every $h \in \mathrm{C}_0(\Omega)$ by (1) and thus $\|t\| = \||t|\|$ for every $t \in \Gamma_s$. If $t, u \in \Gamma_s$ with $|t| \leq |u|$, we find sequences $(f_n)_{n \in \N} \in \mathrm{C}_0(\Omega)$ with $\lim f_ns = t$ and $(g_n)_{n \in \N}$ in $\mathrm{C}_0(\Omega)_+$ with $\lim g_n s = |u| - |t|$. But then
	\begin{align*}
		\|t\| = \||t|\| = \lim_{n \rightarrow \infty} \||f_n|s\| \leq \lim_{n \rightarrow \infty} \|(|f_n| + g_n)s\| = \||u|\| = \|u\|.
	\end{align*}
By Corollary 1.4 and Theorem 2.2 of \cite{MiWo1974}, $\Gamma_s$ is a Banach lattice with positive cone $|\Gamma_s| = \overline{\mathrm{C}_0(\Omega)s}$ and $|\cdot|$ as its modulus, and, if $\K=\C$, that $\Gamma_s$ is the complexification of the real Banach lattice $\overline{\mathrm{C}_0(\Omega,\R)s}$ (cf. Section II.11 of \cite{Scha1974}). In particular, (iv) holds and this implies (ii) and (iii) by the usual formulas for vector lattices (see Corollary 1 of Proposition II.1.4 of \cite{Scha1974}). Finally, if $f \in \mathrm{C}_0(\Omega,\R)$, then $fs \geq 0$ if and only if $|f|s = fs$, i.e., $f-|f| \in I_s$. But by \cref{descriptionsupport} this is exactly the case when $f|_{\supp(s)} \geq 0$, showing (i).
\end{proof}
We use this observation to introduce different types of Banach modules.
\subsection{AM-modules}
The first is based on the concept of AM-spaces (see \cite{Scha1974}, Section II.7).
\begin{definition}
Let $\Omega$ be a locally compact space. A Banach module $\Gamma$ over $\mathrm{C}_0(\Omega)$ is an \emph{AM-module over $\mathrm{C}_0(\Omega)$} if $\Gamma_s$ is an AM-space for each $s \in \Gamma$.
\end{definition}
\begin{remark}
By \cref{modulelattice} a Banach module over $\mathrm{C}_0(\Omega)$ is an AM-module over $\mathrm{C}_0(\Omega)$ if and only if
\begin{align*}
\max(\|f_1s\|,\|f_2s\|) = \|(f_1 \vee f_2)s\|
\end{align*}
for all $f_1, f_2 \in \mathrm{C}_0(\Omega)_+$ and $s \in \Gamma$.
\end{remark}
\begin{example}
If $E$ is a topological Banach bundle over a locally compact space $\Omega$, then $\Gamma_0(E)$ (see \cref{defcontsec}) is an AM-module over $\mathrm{C}_0(\Omega)$.
\end{example}
\begin{remark}\label{amalremarks}
\begin{enumerate}[(i)]
\item AM-modules are known in the literature as \emph{locally convex Banach modules} (see Definition 7.10 in \cite{Gier1998} or Definition 1.1 of \cite{Para2008}, see also \cite{HoKe1977}) and are defined differently. By Proposition 7.14 of \cite{Gier1998} our definition is equivalent in the unital case, and using an approximate identity, even in the general setting.
Our terminology leads to a duality between AM- and AL-modules, see \cref{alvsam} below.
\item Given a compact space $K$, each AM-module over $\mathrm{C}(K)$ is isometrically isomorphic to a space of sections $\Gamma(E)$ of some Banach bundle $E$ over $K$ which is unique up to isometric isomorphy (see Theorems 2.5 and 2.6 of \cite{DuGi1983}). The same holds (and is probably well-known) in the locally compact case if $\Gamma(E)$ is replaced with $\Gamma_0(E)$. However, since we did not find a reference for this fact, we give a proof in \cref{representationbundle} below.
\end{enumerate}
\end{remark}

We now state and prove our first representation result for dynamical Banach modules.
\begin{theorem}\label{main}
Let $G$ be a locally compact group, $S\subset G$ be a closed submonoid and $(\Omega;\varphi)$ a topological $G$-dynamical system. Then the assignments
\begin{align*}
(E;\Phi) &\mapsto (\Gamma_0(E); \mathcal{T}_\Phi)\\
\Theta &\mapsto V_{\Theta}
\end{align*}
define an essentially surjective, fully faithful functor from the category of $S$-dynamical topological Banach bundles over $(\Omega;\varphi)$ to the category of $S$-dynamical AM-modules over $(\mathrm{C}_0(\Omega);T_{\varphi})$.
\end{theorem}
The proof of \cref{main} starts with the following simple observation.
\begin{lemma}\label{inducedbundle}
Let $\Omega$ be a locally compact space, $\varphi \colon \Omega \longrightarrow \Omega$ a homeomorphism and $(E,p_E)$ be a Banach bundle over $\Omega$. Then $(E_\phi,p_\phi)$ with $E_\phi \defeq E$ and $p_\phi \defeq \varphi^{-1} \circ p_E$ is a Banach bundle over $\Omega$ which has the following properties.
\begin{enumerate}[(i)]
\item The identical mapping $\mathrm{id}_E\colon E \longrightarrow E_\phi$ is a Banach bundle morphism over $\phi^{-1}$.
\item If $F$ is a Banach bundle over $\Omega$, then a mapping $\Phi \colon F \longrightarrow E$ is a Banach bundle morphism over $\varphi$ if and only if $\Phi \colon F \longrightarrow E_\varphi$ is a Banach bundle morphism over $\id_{\Omega}$.
\end{enumerate}
\end{lemma}
Using these facts, most of the proof of \cref{main} can be reduced to the non-dynamical case. We first consider single operators.
\begin{lemma}\label{homismorph}
Let $E$ and $F$ be Banach bundles over a locally compact space $\Omega$. Moreover, let $\varphi \colon \Omega \longrightarrow \Omega$ be a homeomorphism and $\mathcal{T} \in \mathscr{L}(\Gamma_0(E),\Gamma_0(F))$ a $T_{\varphi}$-module homomorphism. Then there is a unique Banach bundle morphism $\Phi$ over $\phi$ with $\mathcal{T} = \mathcal{T}_\Phi$. Moreover, $\|\Phi\| = \|\mathcal{T}\|$ and $\mathcal{T}$ is an isometry if and only if $\Phi$ is isometric.
\end{lemma}
\begin{proof}
Assume that $\Omega =K$ is compact. Consider the bundle $F_\phi$ induced by $\varphi$, see \cref{inducedbundle}. The operator $V\in \mathscr{L}(\Gamma(E),\Gamma(F_\phi))$ defined by $Vs \defeq s \circ \varphi$ is an isometric and surjective $T_{\varphi^{-1}}$-homorphism. Therefore, $V\mathcal{T} \in \mathscr{L}(\Gamma(E),\Gamma(F_\phi))$ is a (non-dynamical) homomorphism of Banach modules. 
By Theorem 2.6 of \cite{DuGi1983} we thus find a unique bundle morphism $\Phi \colon E \longrightarrow F_\varphi$ over $\id_K$ with 
\begin{align*}
V\mathcal{T}s  = \Phi \circ s
\end{align*}
for each $s \in \Gamma(E)$, i.e., $\Phi\colon E \longrightarrow F$ is the unique bundle morphism over $\phi$ with 
\begin{align*}
\mathcal{T}s = V^{-1}(\Phi \circ s) = \Phi \circ s \circ \varphi^{-1}
\end{align*}
for every $s \in \Gamma(E)$. Moreover, $\|\Phi\| = \|V\mathcal{T}\| = \|\mathcal{T}\|$ and $\Phi$ is isometric if and only if $V\mathcal{T}$ is an isometry, i.e., if and only if $\mathcal{T}$ is isometric (see Propositions 10.13 and 10.16 of \cite{Gier1998}).\\
Now suppose that $\Omega$ is non-compact, but locally compact. Let $K$ be the one-point compactification and $\tilde{\varphi} \colon K \longrightarrow K$ the canonical continuous extension of $\varphi$. The canonical mapping
\begin{align*}
\Gamma(\tilde{E}) \longrightarrow \Gamma_0(E), \quad s \mapsto s|_\Omega
\end{align*}
is an isometric isomorphism of Banach spaces (see \cref{moduleextension}) and therefore $\mathcal{T}$ induces an operator $\tilde{\mathcal{T}} \in \mathscr{L}(\Gamma(\tilde{E}), \Gamma(\tilde{F}))$. It is easy to check that $\tilde{\mathcal{T}}$ is a $T_{\tilde{\varphi}}$-homomorphism and we can apply the first part to find a unique bundle morphism $\tilde{\Phi}\colon \tilde{E} \longrightarrow \tilde{E}$ over $\tilde{\varphi}$ with $\mathcal{T}(s|_\Omega) = (\tilde{\Phi} \circ s \circ \tilde{\varphi}^{-1})|_\Omega$ for every $s \in \Gamma(\tilde{E})$. Since each Banach bundle morphism of $E$ over $\varphi$ has a unique extension to a Banach bundle morphism of $\tilde{E}$ over $\tilde{\varphi}$ (see \cref{locallycomp}), the restriction $\tilde{\Phi}|_E$ is the unique bundle morphism $\Phi$ over $\varphi$ with $\mathcal{T}s \defeq \Phi \circ s \circ \varphi^{-1}$ for all $s \in \Gamma_0(E)$. The remaining claims are obvious.
\end{proof}
\begin{lemma}\label{esssurj}
Let $G$ be a locally compact group, $S\subset G$ be a closed submonoid and $(\Omega;\varphi)$ a topological $G$-dynamical system. Moreover, let $E$ be a Banach bundle over $\Omega$ and let $\mathcal{T}\colon S \longrightarrow \mathscr{L}(\Gamma_0(E))$ be a strongly continuous monoid homomorphism such that $(\Gamma_0(E);\mathcal{T})$ is an $S$-dynamical Banach module over $(\mathrm{C}_0(\Omega);T_{\phi})$. Then there is a unique $S$-dynamical Banach bundle $(E;\Phi)$ over $(\Omega;\varphi)$ such that $\mathcal{T}_{\Phi} = \mathcal{T}$.
\end{lemma} 
\begin{proof}
We apply \cref{homismorph} to find a unique bundle morphism $\Phi_g$ over $\phi_g$ such that $\mathcal{T}(g) = \mathcal{T}_{\Phi_g}$ for each $g \in S$. Since $\mathcal{T}(1) = \mathrm{Id}_{\Gamma_0(E)}$, we obtain that $\Phi(1) = \mathrm{id}_E$. Moreover, for $g_1,g_2 \in S$ we obtain that $\tilde{\Phi}\defeq \Phi_{g_1} \circ \Phi_{g_2}$ is a bundle morphism over $\varphi_{g_1g_2}$ with 
\begin{align*}
\mathcal{T}(g_1g_2) = \mathcal{T}(g_1)\mathcal{T}(g_2) = \mathcal{T}_\Phi(g_1)\mathcal{T}_\Phi(g_2) = \mathcal{T}_{\tilde{\Phi}}.
\end{align*}
By uniqueness of $\Phi_{g_1g_2}$ we therefore obtain
\begin{align*}
\Phi_{g_1} \circ \Phi_{g_2} = \tilde{\Phi} = \Phi_{g_1g_2}.
\end{align*}
To conclude the proof we have to show that the mapping 
\begin{align*}
\Phi\colon S \longrightarrow E^E, \quad g \mapsto \Phi_g
\end{align*}
is jointly continuous and that $\Phi$ is locally bounded. The latter follows since $\|\Phi(g)\| = \|\mathcal{T}(g)\|$ for every $g \in S$ by \cref{homismorph} and $\mathcal{T}$ is locally bounded by strong continuity and the principle of uniform boundedness.\\
Now let $v \in E$ and $g \in S$. Take $s \in \Gamma_0(E)$ with $s(gp_E(v)) = \Phi_gv$, $\varepsilon > 0$ and an open neighborhood $U$ of $gp_E(v)$. Since $\Phi_g$ is continuous, we find $\tilde{s} \in \Gamma_0(E)$, $\delta > 0$ and a neighborhood $\tilde{V}$ of $p_E(v)$ such that $\tilde{s}(p_E(v)) = v$ and
\begin{align*}
\Phi_g(V(\tilde{s},\tilde{V},\delta)) \subset V(s,U,\varepsilon),
\end{align*}
see \cref{topology}. In particular, we obtain $g(\tilde{V})\subset U$ and $\|\Phi_g \tilde{s}(x) - s(gx)\| < \varepsilon$ for every $x \in \tilde{V}$. Since $\varphi$ is continuous, we find a neigborhood $V \subset \tilde{V}$ of $p_E(v)$ and a neighborhood $\tilde{W}$ of $g$ in $S$ such that $hy \in g(\tilde{V})$ for every $y \in V$ and $h \in \tilde{W}$. Finally, choose a compact neighborhood $W \subset \tilde{W}$ of $g$ with 
\begin{align*}
\sup_{x \in \Omega} \|\Phi_h\tilde{s}(x) - \Phi_g\tilde{s}(g^{-1}hx)\|  = \|\mathcal{T}(h)\tilde{s} - \mathcal{T}(g)\tilde{s}\| < \varepsilon.
\end{align*}
for every $h \in W$. Then $M \defeq \sup_{h \in W} \|\Phi_h\| < \infty$ and for $h \in W$ and $u \in V(\tilde{s}, V, \frac{\varepsilon}{M+1})$, we obtain $hp_E(u) \in U$ and
\begin{align*}
\|\Phi_h u - s(hp_E(u))\| &\leq \|\Phi_h\|\cdot \|u - \tilde{s}(p_E(u))\|\\
&\quad+ \|\Phi_h \tilde{s}(p_E(u)) - \Phi_g \tilde{s}(g^{-1}hp_E(u))\|\\
&\quad+ \|\Phi_g \tilde{s}(g^{-1}hp_E(u)) - s(hp_E(u))\|\\
&< 3\varepsilon.
\end{align*}
This shows $\Phi_h u \in V(s,U, 3\varepsilon)$ for each $h \in W$ and $u \in V(\tilde{s}, V, \frac{\varepsilon}{M+1})$ and thus $\Phi$ is jointly continuous.
\end{proof}
Finally, we look at AM-modules.
\begin{proposition}\label{representationbundle}
Let $\Omega$ be a locally compact space and $\Gamma$ an AM-module over $\mathrm{C}_0(\Omega)$. Then there is a Banach bundle $E$ over $\Omega$ such that $\Gamma_0(E)$ is isometrically isomorphic to $\Gamma$. Moreover, this bundle is unique up to isometric isomorphy.
\end{proposition}
\begin{proof}
If $\Omega$ is compact, the claim holds by Theorem 2.6 of \cite{DuGi1983}. If $\Omega$ is non-compact, we consider $\Gamma$ as a Banach module over $\mathrm{C}(K)$ where $K$ is the one-point compactification of $\Omega$ (see \cref{extendedmodule}). Using a similar argument as in \cref{extendedmodule} we see that $\Gamma$  is then an AM-module over $\mathrm{C}(K)$ and we therefore find a Banach bundle $F$ over $K$ such that $\Gamma(F)$ is isometrically isomorphic to $\Gamma$ as a Banach module over $\mathrm{C}(K)$. Moreover, by the proof of Theorem 2.6 of \cite{DuGi1983} we have $F_\infty \cong \Gamma/J_\infty$ with
\begin{align*}
	J_\infty = \overline{\mathrm{lin}}\{fs\mid f\in \mathrm{C}(K) \textrm{ with } f(\infty) = 0 \textrm{ and } s \in \Gamma\}.
\end{align*}
Since $\Gamma$ is non-degenerate, we obtain $J_\infty = \Gamma$ and thus $F_\infty = \{0\}$. We can therefore define a Banach bundle $E$ over $\Omega$ by setting $E \defeq F \setminus F_\infty$ and $p_E \defeq p_F|_E$ and it is clear that $F= \tilde{E}$. In particular, we obtain an isometric isomorphism of Banach spaces (see \cref{moduleextension})
\begin{align*}
\Gamma(F) \longrightarrow \Gamma_0(E), \quad s \mapsto s|_\Omega
\end{align*}
and it is then easy to check that $\Gamma$ is isometrically isomorphic to $\Gamma_0(E)$ as a Banach module over $\mathrm{C}_0(\Omega)$. Uniqueness up to isometric isomorphy follows directly from \cref{homismorph}.
\end{proof}
Combining \cref{representationbundle} with the preceding Lemmas \ref{homismorph} and \ref{esssurj}, the proof of \cref{main} is straightforward. 
\begin{remark}
It is not hard to construct an inverse to the functor of \cref{main}. In fact, if $\Gamma$ is an AM-module over $\mathrm{C}_0(\Omega)$, then we obtain the fibers $E_x$ of a Banach bundle $E$ by setting
\begin{align*}
&J_x \defeq \overline{\mathrm{lin}}\{fs\mid f\in \mathrm{C}_0(\Omega) \textrm{ with } f(x) = 0 \textrm{ and } s \in \Gamma\},\\
&E_x \defeq \Gamma/J_x,
\end{align*}
for $x \in \Omega$, see Section 2 of \cite{DuGi1983} or Section 7 of \cite{Gier1998}. Moreover, if $\varphi \colon \Omega \longrightarrow \Omega$ is a homeomorphism and $\mathcal{T} \in \mathscr{L}(\Gamma)$ is a $T_\varphi$-homomorphism, then $\mathcal{T}J_x \subset J_{\varphi(x)}$ for every $x \in \Omega$ and therefore $\mathcal{T}$ induces a bounded operator $\Phi_x \in \mathscr{L}(E_x,E_{\varphi(x)})$.\\
With these constructions one can assign a dynamical Banach bundle to a dynamical AM-module $(\Gamma;\mathcal{T})$. We skip the details (cf. Theorem 2.6 of \cite{DuGi1983}).
\end{remark}

\subsection{AL-modules}
The dual concept of AM-spaces in the theory of Banach lattices are so-called AL-spaces (see Section II.8 of \cite{Scha1974}). Again we make use of this concept to introduce a certain class of Banach modules. 
\begin{definition}
Let $\Omega$ be a locally compact space. A Banach module $\Gamma$ over $\mathrm{C}_0(\Omega)$ is called an \emph{AL-module over $\mathrm{C}_0(\Omega)$} if $\Gamma_s$ is an AL-space for each $s \in \Gamma$.
\end{definition}
\begin{remark}
By \cref{modulelattice} a Banach module over $\mathrm{C}_0(\Omega)$ is an AL-module over $\mathrm{C}_0(\Omega)$ if and only if
\begin{align*}
\|f_1s+f_2s\| = \|f_1s\| + \|f_2s\|
\end{align*}
for all $f_1, f_2 \in \mathrm{C}_0(\Omega)_+$ and $s \in \Gamma$.
\end{remark}
Note that if $X$ is a measure space, then $\mathrm{L}^\infty(X)$ is isomorphic to $\mathrm{C}(K)$ as a Banach algebra and a Banach lattice for some compact space $K$. Thus, every Banach module over $\mathrm{L}^\infty(X)$ can be seen as a Banach module over $\mathrm{C}(K)$. In particular, we may speak of \emph{AM- and AL-modules over $\mathrm{L}^\infty(X)$}.
\begin{example}
Let $E$ be a measurable Banach bundle over a measure space $X$. Then $\Gamma^1(E)$ (see \cref{measurablesection}) is an AL-module over $\mathrm{L}^\infty(X)$.
\end{example}
\begin{remark}
It is tempting to expect that for a measure space $X$ every AL-module over $\mathrm{L}^\infty(X)$ is already isomorphic to a space $\Gamma^1(E)$ for some measurable Banach bundle $E$ over $X$. However, we will see below that this is not the case (see \cref{counterexampleAL}).
\end{remark}
As in the case of Banach lattices, AM- and AL-modules are dual to each other. To formulate this result we first equip the dual space of a Banach module with a module structure.
\begin{definition}
Let $K$ be a compact space and $\Gamma$ a Banach module over $\mathrm{C}(K)$. Then the dual space $\Gamma'$ equipped with the operation $(f \cdot s')(s)\defeq s'(f\cdot s)$ for $s \in \Gamma$, $s' \in \Gamma'$ and $f \in \mathrm{C}(K)$ is the \emph{dual Banach module of $\Gamma$ over $\mathrm{C}(K)$.}
\end{definition}
It is straightforward to check that the dual Banach module of a Banach module is in fact a Banach module. We can now make the duality between AM- and AL-modules precise using the following result due to Cunnigham (see Theorem 5 of \cite{Cunn1967}) though in somewhat different notation. 
\begin{proposition}\label{alvsam}
Let $K$ be a compact space. For a Banach module $\Gamma$ over $\mathrm{C}(K)$ the following assertions hold.
\begin{enumerate}[(i)]
\item $\Gamma$ is an AM-module if and only if $\Gamma'$ is an AL-module.
\item $\Gamma$ is an AL-module if and only if $\Gamma'$ is an AM-module.
\end{enumerate}
\end{proposition}

\section{Lattice normed modules}
\subsection{\texorpdfstring{$\mathrm{U}_{0}(\Omega)$-} nnormed modules}
As observed in \cite{Cunn1967}, AM-modules admit an additional lattice theoretic structure. For a locally compact space $\Omega$, we write
\begin{align*}
\mathrm{U}(\Omega)&\defeq \{f \colon \Omega \longrightarrow \R\mid f \textrm{ upper semicontinuous}\},\\
\mathrm{U}_{0}(\Omega) &\defeq \{f \in \mathrm{U}(\Omega)\mid \forall \varepsilon >0 \exists K \subset \Omega \textrm{ compact with } |f(x)| \leq \varepsilon \forall x \notin K\},\\
\mathrm{U}_{0}(\Omega)_+ &\defeq \{f \in \mathrm{U}_0(\Omega)\mid f \geq 0\},
\end{align*}
and introduce the following concept (see Section 6.6 of \cite{HoKe1977} for the compact case).
\begin{definition}
Let $\Omega$ be a locally compact space and $\Gamma$ a Banach module over $\mathrm{C}_0(\Omega)$. A mapping 
\begin{align*}
| \cdot | \colon \Gamma \longrightarrow \mathrm{U}_{0}(\Omega)_+
\end{align*}
is a \emph{$\mathrm{U}_{0}(\Omega)$-valued norm} if
\begin{enumerate}[(i)]
\item $\||s|\| = \|s\|$,
\item $|fs| = |f| \cdot |s|$,
\item $|s_1 + s_2| \leq |s_1| + |s_2|$,  
\end{enumerate}
for all $s,s_1,s_2 \in \Gamma$ and $f \in \mathrm{C}_0(\Omega)$. A Banach module over $\mathrm{C}_0(\Omega)$ together with a $\mathrm{U}_{0}(\Omega)$-valued norm is called a \emph{$\mathrm{U}_{0}(\Omega)$-normed module}.
\end{definition}
\begin{example}
Let $E$ be a Banach bundle over a locally compact space $\Omega$. Setting $|s|(x) \defeq \|s(x)\|$ for $x \in \Omega$ and $s\in \Gamma_0(E)$ turns $\Gamma_0(E)$ into a $\mathrm{U}_{0}(\Omega)$-normed module.
\end{example}
Note that each $\mathrm{U}_{0}(\Omega)$-normed module is automatically an AM-module over $\mathrm{C}_0(\Omega)$. The converse also holds and is basically due to Cunningham in the compact case (see Lemma 3 and Theorem 2 in \cite{Cunn1967}).
	\begin{proposition}\label{vectorvaluednorm1}
	Let $\Omega$ be a locally compact space. For a Banach module $\Gamma$ over $\mathrm{C}_0(\Omega)$ the following are equivalent. 
		\begin{enumerate}[(a)]
		\item $\Gamma$ is an AM-module over $A$.
		\item $\Gamma$ admits a $\mathrm{U}_{0}(\Omega)$-valued norm.
		\end{enumerate} 
	If this assertions hold, then the $\mathrm{U}_{0}(\Omega)$-valued norm is unique and given by
		\begin{align*}
			|s|(x) = \inf \{\|fs\| \mid f \in \mathrm{C}_0(\Omega)_+ \textrm{ with } f(x)=1\} 
		\end{align*}
		for $x \in \Omega$ and $s\in \Gamma$.
	\end{proposition}
	\begin{proof}
		Using \cref{extendedmodule} and an approximate unit, existence via the desired formula of the $\mathrm{U}_0(\Omega)$-valued norm can be reduced to the compact case which is treated in Lemma 3 and Theorem 2 of \cite{Cunn1967}.\\
		For uniqueness, observe that any $\mathrm{U}_{0}(\Omega)$-valued norm $| \cdot | \colon \Gamma \longrightarrow \mathrm{U}_{0}(\Omega)_+$ satisfies 
		\begin{align*}
			|s|(x) \leq \inf \{\|fs\| \mid f \in \mathrm{C}_0(\Omega)_+ \textrm{ with } f(x)=1\} 
		\end{align*}
		for every $x \in \Omega$ and $s \in \Gamma$. On the other hand, if $x \in \Omega$, $s \in \Gamma$ and $\varepsilon>0$, we find a neighborhood $U$ of $x$ such that $|s|(y) \leq |s|(x) + \varepsilon$ for every $y \in U$ since $|s|$ is upper semicontinuous. Thus there is $f \in \mathrm{C}_0(\Omega)_+$ with $\|f\| = f(x)=1$ and 
		\begin{align*}
			\|fs\| = \sup_{y \in \Omega} |fs|(y) = \sup_{y \in \Omega} |f(y)| \cdot |s|(y) \leq |s|(x) + \varepsilon
		\end{align*}		 
		which implies the claim.
	\end{proof}
	\begin{remark}
	The representing Banach bundles of AM-modules $\Gamma$ over $\mathrm{C}_0(\Omega)$ satisfying $|s| \in \mathrm{C}_0(\Omega) \subset \mathrm{U}_0(\Omega)$ for every $s \in \Gamma$ are precisely the continuous Banach bundles (see Theorem 15.11 of \cite{Gier1998} or pages 47--48 of \cite{DuGi1983} for the compact case; the locally compact case can easily be reduced to this). 
	\end{remark}
We can now state the main theorem of this subsection which shows that the algebraic and lattice theoretic structures of $\mathrm{U}_0(\Omega)$-normed modules are closely related to each other. Here, we use the notation $T_\varphi$ for the map $\mathrm{U}_0(\Omega) \longrightarrow \mathrm{U}_0(\Omega), \, f \mapsto f \circ \varphi^{-1}$.
\begin{theorem}\label{latticevsmod}
Let $\Omega$ be a locally compact space, $\varphi \colon \Omega \longrightarrow \Omega$ a homeomorphism and $\Gamma$ and $\Lambda$ $\mathrm{U}_0(\Omega)$-normed modules. For $\mathcal{T} \in \mathscr{L}(\Gamma,\Lambda)$ the following are equivalent.
\begin{enumerate}[(a)]
\item $\mathcal{T}(fs) = T_\varphi f \cdot \mathcal{T}s$ for every $f \in \mathrm{C}_0(\Omega)$ and $s \in \Gamma$.
\item $\supp(\mathcal{T}s) \subset \varphi(\supp(s))$ for every $s \in \Gamma$.
\item $|\mathcal{T}s|\leq \|\mathcal{T}\| \cdot T_\varphi|s|$ for every $s \in \Gamma$.
\item There is $m>0$ such that $|\mathcal{T}s|\leq m \cdot T_\varphi|s|$ for every $s \in \Gamma$.
\end{enumerate}
Moreover, if $\Gamma = \Gamma_0(E)$ and $\Lambda = \Gamma_0(F)$ for Banach bundles $E$ and $F$ over $\Omega$, then the properties above are also equivalent to the following assertion.
\begin{enumerate}[(e)]
\setcounter{enumi}{3}
\item There is a morphism $\Phi$ over $\varphi$ with $\mathcal{T} = \mathcal{T}_\Phi$.
\end{enumerate}
If (e) holds, then the morphism $\Phi$ in (e) is unique, $\|\Phi\| = \|\mathcal{T}\|$ and $\Phi$ is isometric if and only if $\mathcal{T}$ is isometric.
\end{theorem}
For the proof we need the following lemma connecting the vector-valued norm with the concept of support introduced in \cref{defsupport}.
\begin{lemma}\label{suppdesc}
Let $\Gamma$ be a $\mathrm{U}_0(\Omega)$-normed module. Then 
\begin{align*}
\supp(s) = \supp(|s|) = \overline{\{x \in \Omega\mid |s|(x) \neq 0\}}
\end{align*}
for each $s \in \Gamma$.
\end{lemma}
\begin{proof}
Let $x \in \Omega$ with $|s|(x) \neq 0$ and $f \in \mathrm{C}_0(\Omega)$ with $f(x)\neq 0$. Then $|fs|(x) = |f|(x)|s|(x) \neq 0$ and therefore $\|fs\| \neq 0$.\\
Conversely, let $x \in \supp(s)$. Assume there is an open neighborhood $U$ of $x$ such that $|s|(y) = 0$ for every $y \in U$. We then find $f \in \mathrm{C}_0(\Omega)$ with support in $U$ and $f(x) = 1$. But then $|fs| = |f| |s| = 0$ and therefore $fs = 0$ which contradicts $x \in \supp(s)$.
\end{proof}
\begin{proof}[of \cref{latticevsmod}]
The equivalence of (a) and (b) holds by \cref{supportlemma}. Now assume that (a) and (b) hold and that there is $s \in \Gamma$ such that $|\mathcal{T}s| \not\leq \|\mathcal{T}\| \cdot T_\varphi|s|$. We then find $x \in \Omega$ with $\|\mathcal{T}\| \cdot |s|(x) < |\mathcal{T}s|(\varphi(x))$. Since $|s|$ is upper semi-continuous, we find $\varepsilon >0$ and an open neighborhood $V$ of $x$ such that $\|\mathcal{T}\|\cdot|s|(z) \leq |\mathcal{T}s|(\varphi(x)) - \varepsilon$ for all $z \in V$. Now take a function $f \in \mathrm{C}_{0}(\Omega)_+$ with support in $V$ such that $0 \leq f \leq \mathbbm{1}$ and $f(x) = 1$. Setting $\tilde{s} \defeq fs$ we obtain
\begin{align*}
\|\mathcal{T}\|\cdot \|\tilde{s}\| + \varepsilon &= \sup_{z \in V} \|\mathcal{T}\| \cdot f(z)\cdot |s|(z) + \varepsilon  \\&
\leq |\mathcal{T}s|(\varphi(x)) = (T_\varphi f)(\varphi(x))\cdot |\mathcal{T}s|(\varphi(x))
= |\mathcal{T}(fs)|(\varphi(x))\\
&\leq \|\mathcal{T}\tilde{s}\|,
\end{align*} 
which contradicts the definition of $\|\mathcal{T}\|$. The implication \enquote{(c) $\Rightarrow$ (d)} is obvious and \enquote{(d) $\Rightarrow$ (b)} follows from \cref{suppdesc}. The rest of the theorem follows from \cref{homismorph}.
\end{proof}
\begin{remark}
In view of \cref{vectorvaluednorm1} and \cref{latticevsmod}, the assignments of \cref{main} also define an essentially surjective and fully faithful functor from the category of dynamical Banach bundles over a topological dynamical system $(\Omega;\varphi)$ to the category having as objects pairs of $\mathrm{U}_0(\Omega)$-normed modules and monoid representations of \enquote{dominated operators} (in the sense of \cref{latticevsmod} (c)) and as morphisms operators $V \in \mathscr{L}(\Gamma,\Lambda)$ between $\mathrm{U}_0(\Omega)$-normed modules such that there is an $m >0$ with $|Vs| \leq m \cdot |s|$ for all $s \in \Gamma$ which are compatible with the monoid representations.
\end{remark}
\subsection{\texorpdfstring{$\mathrm{L}^1(X)$-} nnormed modules}
AL-modules also admit a vector-valued norm. 
\begin{definition}\label{dualnormed}
Let $\Omega$ be a locally compact space and $\Gamma$ a Banach module over $\mathrm{C}_0(\Omega)$. A mapping 
\begin{align*}
| \cdot | \colon \Gamma \longrightarrow \mathrm{C}_0(\Omega)'_+
\end{align*}
is an \emph{$\mathrm{C}_0(\Omega)'$-valued norm} if
\begin{enumerate}[(i)]
\item $\||s|\| = \|s\|$,
\item $|fs| = |f| \cdot |s|$,
\item $|s_1 + s_2| \leq |s_1| + |s_2|$,  
\end{enumerate}
for all $s,s_1,s_2 \in \Gamma$ and $f \in \mathrm{C}_0(\Omega)$. A Banach module over $A$ together with a $\mathrm{C}_0(\Omega)'$-valued norm is called a \emph{$\mathrm{C}_0(\Omega)'$-normed module}.
\end{definition}
Again the main part of the following result is due to Cunningham in the compact case (see Theorem 4 of \cite{Cunn1967}). We give a new proof in the general case and also provide an explicit formula for the vector-valued norm.
\begin{proposition}\label{alnorm}
Let $\Omega$ be a locally compact space. For a Banach module $\Gamma$ over $\mathrm{C}_0(\Omega)$ the following are equivalent. 
\begin{enumerate}[(a)]
\item $\Gamma$ is an AL-module over $\mathrm{C}_0(\Omega)$.
\item $\Gamma$ admits a $\mathrm{C}_0(\Omega)'$-valued norm.
\end{enumerate} 
If these assertions hold, then the $\mathrm{C}_0(\Omega)'$-valued norm is unique and given by $|s|(f) \defeq \|fs\|$ for all $s \in \Gamma$ and $f \in \mathrm{C}_0(\Omega)_+$.
\end{proposition}
\begin{proof}
It is clear that (b) implies (a) since $\mathrm{C}_0(\Omega)'$ is an AL-space (cf. Proposition 9.1 of \cite{Scha1974}). If (a) holds, we define $|s|(f) = \|fs\|$ for all $s \in \Gamma$ and $f \in \mathrm{C}_0(\Omega)_+$. For every $s \in \Gamma$ the map $|s|\colon \mathrm{C}_0(\Omega)_+ \rightarrow \R_{\geq 0}$ is additive and positively homogeneous and therefore has a unique positive extension $|s| \in A'$ by Lemma 1.3.3 of \cite{Meye1991} (which obviously also holds in the complex case). Now take an approximate unit $(e_i)_{i \in I}$ for $\mathrm{C}_0(\Omega)$. Then
\begin{align*} 
\|s\| = \lim_{i} \|e_i s\| = \lim_{i}|s|(e_i) = \||s|\|.
\end{align*}
It is clear that $|s_1+s_2| \leq |s_1|+|s_2|$ for all $s_1,s_2 \in \Gamma$. Finally, let $f \in \mathrm{C}_0(\Omega)$ and $s \in \Gamma$. Then
\begin{align*}
|fs|(g) = \|gfs\| = \||gf|s\| = |s|(|f|g) = (|f| \cdot |s|)(g)
\end{align*}
for every $g \in \mathrm{C}_0(\Omega)_+$. This shows $|f\cdot s| = |f| \cdot |s|$.\smallskip

To prove uniqueness, let $|\cdot |$ be any $\mathrm{C}_0(\Omega)'$-valued norm on $\Gamma$ and let $(e_i)_{i \in I}$ be an approximate unit for $\mathrm{C}_0(\Omega)$. Then 
\begin{align*}
\|fs\| = \lim_{i} |fs|(e_i) = \lim_i |s|(f e_i) = |s|(f)
\end{align*}
for each $s \in \Gamma$ and $f \in \mathrm{C}_0(\Omega)_+$, showing the claim.
\end{proof}
Given a measure space $X$, we can consider $\mathrm{L}^\infty(X)$ as a space $\mathrm{C}(K)$ for some compact space $K$. If $\Gamma$ is an AL-module over $\mathrm{L}^\infty(X)$, \cref{alnorm} then yields a vector-valued norm $|\cdot| \colon \Gamma \longrightarrow \mathrm{L}^\infty(X)'_+$. On the other hand, if $E$ is a measurable Banach bundle over $X$, then the mapping
\begin{align*}
|\cdot|\colon \Gamma^1(E) \longrightarrow \mathrm{L}^1(X)_+, \quad s \mapsto \|s(\cdot)\|
\end{align*}
satisfies properties (i) -- (iii) of \cref{dualnormed} and since $\mathrm{L}^1(X)$ embeds canonically (as a Banach lattice and as a Banach module over $\mathrm{L}^\infty(X)$) into $\mathrm{L}^\infty(X)'$, this already defines the unique $\mathrm{L}^\infty(X)'$-valued norm. In particular, an AL-module over $\mathrm{L}^\infty(X)$ can only be isometrically isomorphic to $\Gamma^1(E)$ for some measurable Banach bundle $E$ over $X$ if the $\mathrm{L}^\infty(X)'$-valued norm takes values in (the canonical image of) $\mathrm{L}^1(X)$. This is not always the case as the following example shows.
\begin{example}\label{counterexampleAL}
Let $X$ be any measure space and consider $\Gamma\defeq \mathrm{L}^\infty(X)'$ as a Banach module over $\mathrm{L}^\infty(X)$. Then $\Gamma$ is an AL-module over $\mathrm{L}^\infty(X)$ by \cref{alvsam} since $\mathrm{L}^1(X)$ is an AL-module over $\mathrm{L}^\infty(X)$. The usual modulus $|\cdot| \colon \mathrm{L}^\infty(X)' \rightarrow \mathrm{L}^\infty(X)'$ is given by
\begin{align*}
|s|(f) = \sup\{|s(g)| \mid 0 \leq |g| \leq f\}
\end{align*}
for $f \in \mathrm{L}^\infty(X)_+$ and $s \in \mathrm{L}^\infty(X)'$ (see Corollary 1 to Proposition II.4.2 of \cite{Scha1974}). It is easy to see that
\begin{align*}
\sup\{|s(g)| \mid 0 \leq |g| \leq f\} = \sup\{|s(gf)| \mid 0 \leq |g| \leq \mathbbm{1}\} = \|fs\|
\end{align*}
for $f \in \mathrm{L}^\infty(X)_+$ and $s \in \mathrm{L}^\infty(X)'$ and therefore $|\cdot|$ is the unique $\mathrm{L}^\infty(X)'$-valued norm. 
If $\mathrm{L}^1(X)$ is not finite-dimensional, then $\mathrm{L}^1(X)$ is not reflexive (see Corollary 2 of Theorem II.9.9 in \cite{Scha1974}). By Proposition 8.3 (iii) and (v) of  \cite{Scha1974} there are also positive elemnents in $\mathrm{L}^\infty(X)'$ which are not contained in (the canonical image of) $\mathrm{L}^1(X)$, i.e., there is $s \in \Gamma$ with $|s| \in \mathrm{L}^\infty(X)'\setminus \mathrm{L}^1(X)$.
\end{example}
\begin{definition}
Let $X$ be a measure space. An $\mathrm{L}^\infty(X)'$-normed module $\Gamma$ is called an \emph{$\mathrm{L}^1(X)$-normed module} if $|s| \in \mathrm{L}^1(X)$ for every $s \in \Gamma$.
\end{definition}
We now state and prove our second main result. Here a measure space $X$ is \emph{separable} if there is a sequence $(A_n)_{n \in \N}$ of measurable subsets of $\Omega_X$ such that for every $B \in \Sigma_X$ and every $\varepsilon >0$ there is an $n \in \N$ with $\mu_X(A_n \Delta B) < \varepsilon$.
\begin{theorem}\label{main2}
Let $G$ be a (discrete) group, $S\subset G$ be a submonoid and $(X;\varphi)$ a measure preserving $G$-dynamical system with $X$ separable. Then the assignments
\begin{align*}
(E;\Phi) &\mapsto (\Gamma^1(E); \mathcal{T}_\Phi)\\
\Theta &\mapsto V_{\Theta}
\end{align*}
define an essentially surjective, fully faithful functor from the category of $S$-dynamical separable measurable Banach bundles over $(X;\varphi)$ to the category of $S$-dynamical separable $\mathrm{L}^1(X)$-normed modules over $(\mathrm{L}^\infty(X);T_{\varphi})$.
\end{theorem}
We start by showing that separable Banach bundles over separable measure spaces in fact induce separable spaces of sections.
\begin{proposition}\label{lemmasep}
Let $E$ be a separable measurable Banach bundle over a separable measure space $X$. Then $\Gamma^1(E)$ is separable.
\end{proposition}
The proof of the following lemma is based on the proof of Proposition 4.4 of \cite{FeDo1988} (see also Lemma A.3.5 of \cite{DeRe2000} for a similar result).
\begin{lemma}\label{generatedsep}
Let $E$ be a separable Banach bundle over a measure space $X$ and $(s_n)_{n \in \N}$ in $\mathcal{M}_E$ such that $\lin \{s_n(x)\mid n \in \N\}$ is dense in $E_x$ for almost every $x \in \Omega_X$. Then $\lin \{s_n \mid n \in \N\}$ generates $E$, i.e., every $s \in \mathcal{M}_E$ is an almost everywhere limit of a sequence in $\lin\{\mathbbm{1}_As_n\mid A \in \Sigma_X, n \in \N\}$.
\end{lemma}
\begin{proof}
By the set $\{s_n\mid n \in \N\}$ with its linear hull over $\Q$ (if $\K = \R)$ or $\Q + i\Q$ (if $\K = \C$), we may assume that $\{s_n(x)\mid n \in \N\}$ is dense in $E_x$ for almost every $x \in \Omega_X$.
Now let $s \in \mathcal{M}_E$, $\varepsilon > 0$ and set
\begin{align*}
A_n \defeq \{x \in \Omega_X \mid \|s(x) - s_n(x)\| \leq \varepsilon\} \in \Sigma_X
\end{align*}
for every $n \in \N$. Then 
\begin{align*}
\Omega_X \setminus\left(\bigcup_{n \in \N} A_n\right)
\end{align*}
is a nullset. Therefore, $\|s(x) -\tilde{s}(x)\| \leq \varepsilon$ for almost every $x \in \Omega_X$ where
\begin{align*}
\tilde{s}(x) = \begin{cases} s_n(x) &x \in A_n\setminus \bigcup_{k=1}^{n-1} A_{k}, n \in \N,\\
0 &\textrm{else}.
\end{cases}
\end{align*}
Since $\tilde{s}$ is a measurable section with respect to the Banach bundle generated by $\lin \{s_n \mid n \in \N\}$ (see \cref{generatedmeasbundle}), this shows the claim.
\end{proof}
\begin{lemma}\label{intsequence}
Let $E$ be a separable Banach bundle over a measure space $X$. Then there is a sequence $(s_n)_{n \in \N}$ in $\mathcal{M}_E$ such that
\begin{enumerate}[(i)]
\item $\lin \{s_n(x)\mid n \in \N\}$ is dense in $E_x$ for almost every $x \in \Omega_X$,
\item $\mu_X(\{|s_n| \neq 0\}) < \infty$ for every $n \in \N$,
\item $|s_n| = \mathbbm{1}_{\{|s_n| \neq 0\}}$ almost everywhere for every $n \in \N$,
\end{enumerate}
Moreover, for any sequence $(s_n)_{n \in \N}$ in $\mathcal{M}_E$ with properties (i) and (ii), the set
\begin{align*}
\lin \{\mathbbm{1}_{A}s_n\mid A \in \Sigma_X, n \in \N\} \subset\Gamma^1(E)
\end{align*}
is dense in $\Gamma^1(E)$.
\end{lemma}
\begin{proof}
Let $(s_n)_{n \in \N}$ be a sequence in $\mathcal{M}_E$ satisfying (i). Replacing $s_n$ by $\tilde{s}_n$ defined as
\begin{align*}
\tilde{s}_n(x) \defeq \begin{cases} \frac{1}{\|s_n(x)\|}s_n(x) & s_n(x) \neq 0,\\
0 & s_n(x) = 0,
\end{cases}
\end{align*}
for every $n \in \N$ we may assume that (i) and (iii) hold. Now pick a sequence $(A_n)_{n \in \N}$ of measurable subsets of $\Omega_X$ of finite measure such that 
\begin{align*}
	\Omega_X = \bigcup_{m \in \N} A_m.
\end{align*}
Then $\mu_X(\{|\mathbbm{1}_{A_m}s_n|\neq 0\}) < \infty$ for all $m,n \in \N$. Replacing $(s_n)_{n \in \N}$ once again, we may assume that properties (i) -- (iii) are fulfilled.\\
Now assume that $(s_n)_{n \in \N}$ is a sequence $\mathcal{M}_E$ satisfying (i) and (ii) and let $s \in \mathcal{M}_E$ with $\int |s|\, \mathrm{d}\mu_X < \infty$. By \cref{generatedsep} and Lemma 4.3 of \cite{FeDo1988} we find a sequence $(t_n)_{n \in \N}$ in 
\begin{align*}
M \defeq \lin \{\mathbbm{1}_{A}s_n\mid A \in \Sigma_X, n \in \N\} \subset \mathcal{M}_E
\end{align*}
such that $\lim_{n \rightarrow \infty}t_n =  s$ almost everywhere and $|t_n| \leq |s|$ almost everywhere for all $n \in \N$. By Lebesgue's theorem we therefore obtain that the canonical image of $M$ in $\Gamma^1(E)$ is dense in $\Gamma^1(E)$.
\end{proof}

\begin{proof}[of \cref{lemmasep}]
Using the separability of $X$, we pick a sequence $(A_n)_{n \in \N}$ of measurable subsets of $\Omega_X$ such that for every $B \in \Sigma_X$ and every $\varepsilon >0$ there is $n \in \N$ with $\mu_X(A_n \Delta B) < \varepsilon$. Moreover, take a sequence $(s_n)_{n \in \N}$ as in \cref{intsequence}. For each $n \in \N$ and every $A \in \Sigma_X$ we then find an $m \in \N$ with
\begin{align*}
\|\mathbbm{1}_{A}s_n - \mathbbm{1}_{A_m}s_n\| \leq \mu(A \Delta A_m) < \varepsilon.
\end{align*}
This implies that $\{\mathbbm{1}_{A_m}s_n\mid n,m \in \N\}$ is total in $\Gamma^1(E)$.
\end{proof}
The following result characterizes weighted Koopman operators induced by measurable dynamical Banach bundles similarly to the topological setting (cf. \cref{latticevsmod}).
\begin{theorem}\label{preparationAL}
Let $X$ be a measure space, $\varphi \colon X \longrightarrow X$ an automorphism and $\Gamma$ and $\Lambda$ $\mathrm{L}^1(X)$-normed modules. For an operator $\mathcal{T} \in \mathscr{L}(\Gamma,\Lambda)$ the following are equivalent.
\begin{enumerate}[(a)]
\item $\mathcal{T}(fs) = T_\varphi f \cdot \mathcal{T}s$ for all $f \in \mathrm{L}^\infty(X)$ and every $s \in \Gamma$.
\item $|\mathcal{T}s| \leq \|\mathcal{T}\| \cdot T_\varphi|s|$ for every $s \in \Gamma$.
\item There is an $m > 0$ such that $|\mathcal{T}s| \leq m \cdot T_\varphi|s|$ for every $s \in \Gamma$.
\end{enumerate}
Moreover, if $\Gamma = \Gamma^1(E)$ and $\Lambda = \Gamma^1(F)$ for Banach bundles $E$ and $F$ over $X$ with $E$ separable, then the above are also equivalent to the following assertion.
\begin{enumerate}[(a)]
\setcounter{enumi}{3}
\item There is a morphism $\Phi \colon E\longrightarrow F$ over $\varphi$ such that $\mathcal{T} = \mathcal{T}_\Phi$.
\end{enumerate}
If (d) holds, then the morphism $\Phi$ in (d) is unique, 
\begin{align*}
|\Phi|\colon \Omega_X \longrightarrow [0,\infty), \quad x \mapsto \|\Phi_x\|
\end{align*}
defines an element of $\mathrm{L}^\infty(X)$ and
\begin{itemize} 
\item $\sup\{|\mathcal{T}_\Phi s|\mid s \in \Gamma^\infty(E) \textrm{ with } |s| \leq \mathbbm{1}\} = T_\varphi|\Phi| \in \mathrm{L}^\infty(X)$,
\item $\|\Phi\| = \|\mathcal{T}_\Phi\|_{\Gamma^\infty(E)} = \|\mathcal{T}_\Phi\|_{\Gamma^1(E)}$,
\item $\Phi$ is an isometry if and only $\mathcal{T}_\Phi \in \mathscr{L}(\Gamma^1(E),\Gamma^1(F))$ is an isometry.
\end{itemize}
\end{theorem}
\begin{proof}
We write $\langle\, \cdot \,,\,\cdot\, \rangle$ for the canonical duality between $\mathrm{L}^1(X)$ and $\mathrm{L}^\infty(X)$.
Now assume that (a) is valid and take $s \in \Gamma$. For each $f \in \mathrm{L}^\infty(X)$ with $f \geq 0$ we obtain
\begin{align*}
\langle |\mathcal{T}s|, f\rangle &=  \|f\mathcal{T}s\| = \|\mathcal{T}((T_{{\varphi}^{-1}} f) \cdot  s)\|\\
&\leq \|\mathcal{T}\| \cdot \|T_{{\varphi}^{-1}} f \cdot  s\| = \|\mathcal{T}\| \cdot \langle |s|,T_{{\varphi}^{-1}} f \rangle = \langle \|\mathcal{T}\|\cdot  T_{\varphi}|s|, f \rangle
\end{align*}
since $\varphi$ is measure-preserving. Thus, $|\mathcal{T}s| \leq \|\mathcal{T}\| \cdot T_{\varphi}|s|$.\smallskip

The implication \enquote{(b) $\Rightarrow$ (c)} is clear. Now assume that (c) holds. Since $X$ is $\sigma$-finite, we find measurable and pairwise disjoint sets $A_n \in \Sigma_X$ with finite measure for $n\in \N$ such that 
\begin{align*}
\Omega_X = \bigcup_{n \in \N} A_n.
\end{align*}
For fixed $n \in \N$ consider the submodules
\begin{align*}
\Gamma_n &\defeq \{s \in \Gamma\mid \mathbbm{1}_{A_n}|s| =  |s| \in \mathrm{L}^\infty(X)\} \subset \Gamma,\\
\Lambda_n &\defeq \{s \in \Lambda \mid \mathbbm{1}_{\varphi(A_n)} |s| = |s| \in \mathrm{L}^\infty(X)\}\subset \Lambda.
\end{align*}
We define $\|s\|_{\infty} \defeq \| |s| \|_{\mathrm{L}^\infty(X)}$ for $s \in \Gamma_n$ and $s \in \Lambda_n$, respectively. We show that this turns $\Gamma_n$ and $\Lambda_n$ into Banach modules over $\mathrm{L}^\infty(X)$. If $(s_m)_{n \in \N}$ is a Cauchy sequence in $\Gamma_n$ with respect to the norm $\|\cdot\|_\infty$, then it is also a Cauchy sequence with respect to the norm of $\Gamma$. By completeness of $\Gamma$ there is $s \in \Gamma$ such that $\lim_{n \rightarrow \infty} s_m = s$ in $\Gamma$. Using that there is a subsequence $(s_{m_k})_{k \in \N}$ of $(s_m)_{m \in \N}$ such that $|s_{m_k} - s| \rightarrow 0$ almost everywhere, it follows that $s \in \Gamma_n$ and $\lim_{m \rightarrow \infty} s_m = s$ with respect to $\|\cdot\|_{\infty}$. Thus, $\Gamma_n$---and likewise $\Lambda_n$---is a Banach module over $\mathrm{L}^\infty(X)$. Moreover, $\mathcal{T}|_{\Gamma_n} \in \mathscr{L}(\Gamma_n,\Lambda_n)$ by (c). Choose a compact space $K$ and an isomorphism $V \in \mathscr{L}(\mathrm{L}^\infty(X),\mathrm{C}(K))$ of Banach algebras and lattices. We then consider $\Gamma_n$ and $\Lambda_n$ as Banach modules over $\mathrm{C}(K)$ via $V^{-1}$ and see that the mappings
\begin{align*}
&\Gamma_n \longrightarrow \mathrm{C}(K), \quad s \mapsto V|s|,\\
&\Lambda_n \longrightarrow \mathrm{C}(K), \quad s \mapsto V|s|
\end{align*}
turn $\Gamma_n$ and $\Lambda_n$ into $\mathrm{U}(K)$-normed modules. Moreover, since every algebra isomorphism on $\mathrm{C}(K)$ is induced by a homeomorphism on $K$, we can apply \cref{latticevsmod} to the $VT_\varphi V^{-1}$-homomorphism $\mathcal{T}|_{\Gamma_n} \in \mathscr{L}(\Gamma_n,\Lambda_n)$. This shows that $\mathcal{T}(fs) = (T_\varphi f) \cdot \mathcal{T}s$ for all $f \in \mathrm{L}^\infty(X)$ and $s \in \Gamma_n$.\\
Take $f \in \mathrm{L}^\infty(X)$ and $s \in \Gamma$ with $|s| = \mathbbm{1}_{A_n}|s|$. Then $s = \lim_{m \rightarrow \infty} \mathbbm{1}_{\{|s| \leq m\}}s$ in $\Gamma$ and therefore
\begin{align*}
\mathcal{T}(fs) = \lim_{m \rightarrow \infty} \mathcal{T}(f \mathbbm{1}_{\{|s| \leq m\}}s) = \lim_{m \rightarrow \infty} (T_\varphi f) \cdot \mathcal{T}(\mathbbm{1}_{\{|s| \leq m\}}s) =  (T_\varphi f) \cdot \mathcal{T}s. 
\end{align*}
Finally, we obtain for arbitrary $s \in \Gamma$ and $f \in \mathrm{L}^\infty(X)$
\begin{align*}
\mathcal{T}(fs) &= \lim_{N \rightarrow \infty} \mathcal{T}\left(f\sum_{n =1}^N \mathbbm{1}_{A_n}s\right) =  \lim_{N \rightarrow \infty} \sum_{n =1}^N \mathcal{T}(f \mathbbm{1}_{A_n}s)\\
&= \lim_{N \rightarrow \infty} \sum_{n =1}^N T_\varphi f \cdot \mathcal{T}\mathbbm{1}_{A_n}s = T_\varphi f \cdot \mathcal{T}s.
\end{align*}
This shows (a).\medskip

Now assume that $\Gamma= \Gamma^1(E)$ and $\Lambda = \Gamma^1(F)$ for measurable Banach bundles $E$ and $F$ over $X$ with $E$ separable. We let $Q \defeq \Q$ if $\K= \R$ and $Q \defeq \Q + i\Q$ if $\K = \C$. Now take a sequence $(s_n)_{n \in \N}$ as in $\mathcal{M}_E$ satisfying conidtions (i) -- (iii) of \cref{intsequence} and set
\begin{align*}
H_x \defeq \lin_{Q}\{s_k(x)\mid k \in \N\}
\end{align*}
for every $x \in \Omega_X$.\\
Let $\mathcal{T}$ be a $T_\varphi$-homomorphism. Choose a representative for $\varphi$ (which we again denote by $\varphi$) and a representative $t_n \in \mathcal{M}_F$ of $\mathcal{T}s_n \in \Gamma^1(F)$ for each $n \in \N$. By (b) we obtain 
	\begin{align}
		\left\| \left(\sum_{k=1}^N q_k t_k\right)(\varphi(x))\right\| \leq \|\mathcal{T}\| \cdot \left\| \left(\sum_{k=1}^N q_k s_k\right)(x)\right\| 
	\end{align}
for all $(q_1,...,q_N) \in Q^N$, $N \in \N$ and almost every $x \in \Omega_X$.  For almost every $x \in \Omega_X$ we therefore find a unique $Q$-linear map $\Phi_x \colon H_x \longrightarrow H_{\varphi(x)}$ such that $\Phi_xs_n(x) = (t_n)(\varphi(x))$ for every $n \in \N$. By (3) and property (i) of \cref{intsequence} it has a unique extension to a bounded operator $\Phi_x \in \mathscr{L}(E_x,F_{\varphi(x)})$ for almost every $x \in \Omega_X$. We set $\Phi_x \defeq 0 \in \mathscr{L}(E_x,F_{\varphi(x)})$ for the remaining points $x \in \Omega_X$ and obtain a mapping
\begin{align*}
\Phi \colon E \longrightarrow F, \quad v \mapsto \Phi_{p_E(v)}v.
\end{align*}
Since $\Phi \circ (\mathbbm{1}_A \cdot s_n) = (\mathbbm{1}_{\varphi(A)} \cdot t_n) \circ \varphi$ almost everywhere for every $n \in \N$ and every set $A \in \Sigma_X$, we can apply \cref{generatedsep} to see that for each $s \in \mathcal{M}_E$ there is a $t \in \mathcal{M}_F$ with $\Phi \circ s = t \circ \varphi$ almost everywhere. This shows that $\Phi$ defines a morphism of measurable Banach bundles over $\varphi$ and we denote this again by $\Phi$. Moreover, $\mathcal{T}_\Phi s_n = \mathcal{T}s_n$ and, since $\{\mathbbm{1}_As_n\mid n \in \N, A \in \Sigma_X\}$ defines a total subset of $\Gamma^1(E)$ by \cref{intsequence}, we obtain $\mathcal{T} = \mathcal{T}_\Phi$. Thus (a), (b) and (c) imply (d). The converse implication is obvious.\medskip

Now let $\Phi\colon E \longrightarrow F$ be a morphism over $\varphi$. As usual, we pick a representing premorphism whenever necessary. Using \cref{intsequence} and standard arguments we find a sequence $(\tilde{s}_n)_{n \in \N}$ in $\mathcal{M}_E$ such that 
\begin{itemize}
\item $|\tilde{s}_n|\leq \mathbbm{1}$ almost everywhere for every $n \in \N$,
\item $\mu_X(\{|\tilde{s}_n|\neq 0\}) < \infty$ for every $n \in \N$,
\item $\{\tilde{s}_n(x)\mid n \in \N\}$ is dense in the unit ball $\mathrm{B}_{E_x}$ of $E_x$ for almost every $x \in \Omega_X$.
\end{itemize}
Then
\begin{align*}
\|\Phi|_{E_x}\| = \sup_{n \in \N} \|\Phi|_{E_x}\tilde{s}_n(x)\|
\end{align*}
for almost every $x \in \Omega_X$. Thus, $\Omega_X \longrightarrow \R, \, x \mapsto \|\Phi|_{E_x}\|$ is measurable and $|\Phi|$ defines an element of $\mathrm{L}^\infty(X)$ of norm $\|\Phi\|$.\\
Clearly, $|\mathcal{T}_\Phi s|\leq T_\varphi|\Phi|$ for every $s \in \Gamma^\infty(E)$ with $|s| \leq \mathbbm{1}$.
On the other hand, 
\begin{align*}
T_\varphi|\Phi|(x) &= \|\Phi|_{E_{\varphi^{-1}(x)}}\| = \sup_{n \in \N} \|\Phi|_{E_{\varphi^{-1}(x)}} \tilde{s}_n(\varphi^{-1}(x))\| = \sup_{n \in \N} \|(\mathcal{T}_\Phi \tilde{s}_n)(x)\|
\end{align*}
for almost every $x \in \Omega_X$. This shows
	\begin{align}
		T_\varphi|\Phi| = \sup\{|\mathcal{T}_\Phi s|\mid s \in \Gamma^\infty(E) \textrm{ with } |s| \leq \mathbbm{1}\}.
	\end{align}
Moreover,
\begin{align*}
\|\Phi\| &= \esssup_{x \in \Omega_X} \sup_{n \in \N} \|(\mathcal{T}_\Phi \tilde{s}_n)(x)\| = \sup_{n \in \N} \esssup_{x \in \Omega_X} \|(\mathcal{T}_\Phi \tilde{s}_n)(x)\| \\
&= \sup_{n \in \N} \|\mathcal{T}_\Phi \tilde{s}_n\|_{\Gamma^\infty(E)} \leq \|\mathcal{T}_\Phi\|_{\Gamma^\infty(E)},
\end{align*}
and $\|\mathcal{T}_\Phi\|_{\Gamma^\infty(E)} \leq \|\Phi\|$ is clear, hence $\|\mathcal{T}_\Phi \|_{\Gamma^\infty(E)} = \|\Phi\|$.\\
Now pick $s \in \Gamma^\infty(E)$ with $|s|\leq \mathbbm{1}$. For every measurable set $A \in \Sigma_X$ with finite measure
\begin{align*}
\mathbbm{1}_A|\mathcal{T}_\Phi s| = |\mathcal{T}_\Phi (T_\varphi^{-1}\mathbbm{1}_A \cdot s)| \leq \|\mathcal{T}_\Phi\|_{\Gamma^1(E)}\cdot T_\varphi|(T_\varphi^{-1}\mathbbm{1}_A \cdot s)| \leq \|\mathcal{T}_\Phi\|_{\Gamma^1(E)} \cdot \mathbbm{1}_A
\end{align*}
by (b). Since $X$ is $\sigma$-finite, we obtain $\|\Phi\| \leq \|\mathcal{T}_\Phi\|_{\Gamma^1(E)}$ by (4), and the inequality $\|\mathcal{T}_\Phi\|_{\Gamma^1(E)} \leq \|\Phi\|$ is obvious. Therefore,
\begin{align*}
\|\Phi\| = \|\mathcal{T}_\Phi\|_{\Gamma^\infty(E)} = \|\mathcal{T}_\Phi\|_{\Gamma^1(E)}
\end{align*}
and, since the difference of premorphisms over $\varphi$ is again a premorphism over $\varphi$, this equality also proves the uniqueness of $\Phi$ in (d).\\
If $\Phi$ is an isometry, then clearly $\mathcal{T}_\Phi \in \mathscr{L}(\Gamma^1(E),\Gamma^1(F))$ is an isometry. Assume conversely that $\mathcal{T}_\Phi$ is an isometry. We already know that $\Phi|_{E_x}$ is a contraction for almost every $x \in \Omega_X$. Assume that there is a set $A \in \Sigma_X$ with positive measure such that $\Phi|_{E_x}$ is not an isometry for every $x \in A$. We then find an $n \in \N$ and a set $B \in \Sigma_X$ with positive measure such that $\|\Phi|_{E_x}\tilde{s}_n(x)\| < \|\tilde{s}_n(x)\|$ for every $x \in B$. This implies
\begin{align*}
\|\mathcal{T}_\Phi \tilde{s}_n\| = \int_X \|\Phi|_{E_x}\tilde{s}_n(x)\|\, \mathrm{d}\mu_X < \int_X \|\tilde{s}_n(x)\|\, \mathrm{d}\mu_X = \|\tilde{s}_n\|,
\end{align*}
a contradiction.
\end{proof}
Since we have not employed any continuity assumptions on dynamical measurable Banach bundles, we immediately obtain the following consequence of \cref{preparationAL}.
\begin{corollary}\label{cormeas}
Let $G$ be a (discrete) group, $S\subset G$ be a submonoid and $(X;\varphi)$ a measure-preserving $G$-dynamical system. Moreover let $E$ be a separable Banach bundle over $X$ and let $\mathcal{T}\colon S \longrightarrow \mathscr{L}(\Gamma^1(E))$ be a monoid homomorphism such that $(\Gamma^1(E);\mathcal{T})$ is an $S$-dynamical Banach module over $(\mathrm{L}^\infty(X);T_{\phi})$. Then there is a unique dynamical Banach bundle $(E;\Phi)$ over $(X;\varphi)$ such that $\mathcal{T}_{\Phi} = \mathcal{T}$.
\end{corollary}
Finally, we use a result of Gutmann (\cite{Gutm1993b}) to represent $\mathrm{L}^1(X)$-normed modules. 
\begin{proposition}\label{representationAL}
Let $X$ be a measure space and $\Gamma$ an $\mathrm{L}^1(X)$-normed module. Then the following assertions hold.
\begin{enumerate}[(i)]
\item There is a measurable Banach bundle $E$ over $X$ such that $\Gamma^1(E)$ is isometrically isomorphic to $\Gamma$.
\item If $\Gamma$ is separable, then there is a separable Banach bundle $E$ over $X$ such that $\Gamma^1(E)$ is isometrically isomorphic to $\Gamma$. Moreover, $E$ is unique up to isometric isomorphy.
\end{enumerate}
\end{proposition}
\begin{proof}
In the real case, 7.1.3 of \cite{Kusr2000} shows that the space $\Gamma$ is in particular a Banach--Kantorovich space over $\mathrm{L}^1(X)$ (see Chapter 2 of \cite{Kusr2000} for this concept) and we find a measurable Banach bundle $E$ over $X$ such that $\Gamma$ is isometrically isomorphic to $\Gamma^1(E)$ as a lattice normed space by Theorem 3.4.8 of \cite{Gutm1993b}\footnote{Note that the definition of measurable Banach bundles by Gutmann slightly differs from ours. However, every measurable Banach bundle in the sense of Gutmann canonically defines a measurable Banach bundle in our sense having the same space $\Gamma^1(E)$.}. If we start with a complex $\mathrm{L}^1(X)$-normed module, the proof of this theorem reveals that the constructed Banach bundle $E$ is canonically a Banach bundle of complex Banach spaces and that the isomorphism of $\Gamma$ and $\Gamma^1(E)$ is $\C$-linear (see Theorem 3.3.4 of \cite{Gutm1993b} and Theorems 2.1.5 and 2.4.2 of \cite{Gutm1993a}). In any case, we can apply \cref{preparationAL} to see that this isomorphism is an isometric Banach module isomorphism.\\
Now assume that $\Gamma$ and therefore $\Gamma^1(E)$ is separable. Let $(s_n)_{n \in \N}$ be dense in $\Gamma^1(E)$ and choose a representative in $\mathcal{M}_E$ for each $s_n$ which we also denote by $s_n$. We define a new measurable Banach bundle by setting $F_x \defeq \overline{\lin \{s_n(x)\mid n \in \N\}}$ for every $x \in \Omega_X$ and 
\begin{align*}
\mathcal{M}_F \defeq \{s \in \mathcal{M}_E \mid s(x) \in F_x \textrm{ for every } x \in \Omega_X\}.
\end{align*} 
Then 
\begin{align*}
V \colon \Gamma^1(F) \longrightarrow \Gamma^1(E), \quad s \mapsto s
\end{align*}
is an isometric module homomorphism. However, since $s_n \in \Gamma^1(F)$ for every $n \in \N$, $V$ is in fact an isometric isomorphism. Clearly, $F$ is separable. Uniqueness up to isometric isomorphy follows immediately from \cref{preparationAL}.
\end{proof}
Combining \cref{lemmasep}, \cref{cormeas}, \cref{preparationAL} and \cref{representationAL} now readily yields \cref{main2}.

\begin{remark}
Note that in contrast to the topological setting, the construction of the representing separable measurable Banach bundle is not canonical and involves choices.
\end{remark}

\parindent 0pt
\parskip 0.5\baselineskip
\setlength{\footskip}{4ex}
\bibliographystyle{alpha}
\bibliography{literature} 
\footnotesize

  \textsc{Henrik Kreidler, Mathematisches Institut, Bergische Universität Wuppertal, Gaußstraße 20,
D-42119 Wuppertal, Germany}\par\nopagebreak
  \textit{E-mail address}: \texttt{kreidler@uni-wuppertal.de}\par\nopagebreak
  \textsc{Sita Siewert, Mathematisches Institut, Universität Tübingen, Auf der Morgenstelle 10, D-72076 Tübingen, Germany}\par\nopagebreak
  \textit{E-mail address}: \texttt{sisi@fa.uni-tuebingen.de}
\end{document}